%% file: main.tex
\documentclass[onefignum,onetabnum]{siamart171218}

\input{ex_shared}

\ifpdf
\hypersetup{
  pdftitle={Sparse Separable Bregman},
  pdfauthor={F. S. Akta\c{s} and M. \c{C}  P{\i}nar}
}
\fi

\begin{document}
\maketitle

\begin{abstract}

This paper considers the minimization of a continuously differentiable function over a cardinality constraint. We focus on smooth and relatively smooth functions. These smoothness criteria result in new descent lemmas. Based on the new descent lemmas, novel optimality conditions and algorithms are developed, which extend the previously proposed hard-thresholding algorithms. We give a theoretical analysis of these algorithms and extend previous results on properties of iterative hard thresholding-like algorithms. In particular, we focus on the weighted $\ell_2$ norm, which requires efficient solution of convex subproblems. We apply our algorithms to compressed sensing problems to demonstrate the theoretical findings and the enhancements achieved through the proposed framework.

\end{abstract}


 
\section{Introduction}
Sparsity has been a key concept in signal processing, applied mathematics, statistics, and computer science for a variety of tasks, including compression, denoising, model selection, and image processing, see e.g., \cite{Blumensath2012CompressedSW,donoho2006compressed,yuan_grad_hard_threshold}. In recent years, there has been a surge of interest in sparsity-based methods and algorithms for sparse recovery \cite{amir_beck_sparse_opt,berg-friedlander,foucart-siam,nhtp}. Although exploiting sparsity in different applications has garnered significant attention, most research has focused on recovering sparse data represented by a vector \( x \in \mathbb{R}^n \) from linear measurements of the form \( b = Ax \). For instance, the field of compressed sensing  deals with recovering a sparse vector \( x \) from a small set of linear measurements \( b \in \mathbb{R}^m \), where \( m \) is typically much smaller than \( n \). Given that real-world measurements are often noisy, the standard approach to recovering \( x \) involves finding a sparse vector \( x \) that minimizes the quadratic function \( \|Ax - b\|_2^2 \). These considerations motivate us to study the following optimization model

\begin{equation}
\label{eq:optimization_model}
\begin{array}{ll@{}ll}
\underset{x}\min \quad & f(x)  \\
\text{st.} \quad & ||x||_0 \leq s, \\
\end{array}
\end{equation}
where $f: \mathbb{R}^n \mapsto \mathbb{R}$ is a continuously differentiable convex function and $||x||_0$ is the cardinality function.

The purpose of the present paper is to establish necessary optimality conditions for problem (\ref{eq:optimization_model}) based on the
concept of $H$-stationarity or Bregman stationarity, an extension of the concept of $L$-stationarity of \cite{amir_beck_sparse_opt}, and then propose efficient algorithms for sparse recovery based on the concept of $H$-stationarity. In particular, we relate Bregman stationarity to coordinatewise (CW) optimality. We demonstrate that CW optimality implies Bregman stationarity. Algorithms akin to iterative hard thresholding are devised and convergence properties are analyzed. Careful implementations and extensive numerical experiments support the usefulness of the proposed ideas.
 
\subsection{Related Work}
There is abundant literature on both sparsity constrained problems and optimization using Bregman functions. For sparsity constrained problems, our desktop reference is \cite{amir_beck_sparse_opt} where the authors established first-order necessary optimality conditions for (\ref{eq:optimization_model}). Then, they used these conditions to derive numerical algorithms for computing points satisfying the derived optimality criteria. Different stationarity concepts were derived in \cite{pan-xiu-zhou} using normal and tangent cones to the sparsity constrained set as well as first and second-order optimality conditions and relations between the different notions of stationarity. Necessary optimality conditions under Robinson's constraint qualification were studied in \cite{lu-zhang} for  general sparsity constrained nonlinear programming. The authors also proposed a penalty decomposition method for the solution of the problem. The second-order tangent set to the sparsity constrained set was studied in \cite{pan-luo-xiu} as well as second-order optimality conditions.
An important and well-studied instance of the problem (that we shall study in section 6 as well) occurs in the field of compressed sensing which deal with the recovery of a sparse vector from linear measurements. When measurement noise is present, one deals with the objective function 
$f_{LS}(x)=\|A x - b\|_2$ in (\ref{eq:optimization_model}). 
There exists an entire line of research for the afore-mentioned sparse optimization problem with $f_{LS}$ based on  approximation using $\ell_1$-norm and/or $\ell_2$ norm. A recent paper \cite{xie-li-liang} on sparse vector recovery using a combination of $\ell_1$ and $\ell_2$ norms contains an extensive list of references pertaining to that line of literature. Replacing the $\ell_0$-norm by the $\ell_1$-norm allows exact sparse recovery if and and only if the matrix $A$ satisfies the so-called null-space property \cite{foucart-rauhut}. Other sufficient conditions for exact recovery include the restricted isometry property (RIP) \cite{candesrip}, and coherence property \cite{candes-eldar-needell}.

A key idea related to the algorithms of the present paper is the iterative hard thresholding (IHT) algorithm  proposed and developed in \cite{blumsignal,blumdaviesharmonic,blumdavies-selected,Blumensath-davies}. This development led to a flurry of activities around the theme of IHT
including its notable extensions to a class of cone constrained convex optimization problems regularized by the $\ell_0$-norm in \cite{lu-cone} and to proximal IHT algorithm \cite{prox-iht}. More recent developments along this line include
\cite{amir_beck_sparse_opt,gpnp,nhtp}. The reference lists of the last two references contain pointers to the literature that cannot be reviewed in the present paper for lack of space.

Another  theme that is touched by the present paper is the widespread use of Bregman distances in optimization across various contexts. For foundational work on proximal Bregman-based methods, we direct readers to the references \cite{bolteteboulledynamic,censorzenios,chen_teboulle1993,eckstein,teboullemor} providing  numerous intriguing results that include linking Bregman proximal distance with dynamical systems. For a more in-depth exploration of properties and applications of Bregman monotone algorithms in Banach spaces, we refer the reader to the seminal work \cite{bbcsicopt,Bauschke2017}.

\subsection{Notation}
The set of $s$-sparse vectors is denoted $C_s$, i.e.
\begin{equation}
\label{eq:sparse_set}
\begin{array}{ll@{}ll}
C_s = \{x \in \mathbb{R}^n | ||x||_0 \leq s \}.
\end{array}
\end{equation}
For a given vector $x \in \mathbb{R}^n$, and a set $S \subseteq \{1,2,..,n\}$, subvector induced by $S$ is denoted with $x_S$. Similarly, $A_S$, denotes the square submatrix indexed by $S$. $||x||_A = \sqrt{x^TAx}$ denotes the weighted $\ell_2$ norm. $A_i$ represents $i$'th row of the matrix as a column vector and $A_{ij}$ corresponds to the element in the position of $i$th row and $j$th column of matrix $A$. Vector $e_j$ represents $j$th column of the identity matrix, $\langle ,\rangle$ denotes the dot product, and $\tr$ denotes the trace of the matrix. 

\subsection{Contributions and Organization of the Paper}
In this paper, our primary goal is to develop novel necessary conditions for (\ref{eq:optimization_model}) that will allow the development of new algorithms based on these conditions. Out

\begin{itemize}
    \item In the next section, we briefly go over smoothness and smooth adaptability which play an essential role in the derivation of the necessary conditions and the analysis of the algorithms.
    \item We generalize the classical necessary optimality condition which is derived under the Lipschtiz continuous gradient assumption, by utilizing separable Bregman functions in Section \ref{sec:bregman_stationarity}.
    \item In Section \ref{sec:convergence}, we derive new algorithms based on the novel necessary conditions derived. We analyze the theoretical properties of the algorithms and give new convergence results that show the IWHT algorithm behaves almost like the classical gradient descent or the Bregman proximal gradient algorithm. 
    \item In Section \ref{sec:compute_weights}, we propose three Semidefinite Programming (SDP) models for computing weights that yield new necessary conditions. We apply the BCM algorithm proposed in \cite{cnv_BM_SDP} and propose an extension which we numerically justify in the following section. 
    \item Numerical results in Section \ref{sec:numerical_results} illustrates the theoretical findings and benefit of using novel necessary conditions on sparsity constrained linear system solving. 
\end{itemize}
In addition, several results are relegated into an appendix to improve the readability of the paper.

\section{Preliminaries}
In this section, we define mathematical concepts often used in the characterization of optimal points and convergence analysis of the algorithms. 

\begin{definition} Let $f$ be a continuously differentiable function. $f$ is called \Ls-smooth if its gradient is a Lipschitz continuous with parameter $L_f$:
\begin{equation}
\label{eq:lipschitz_gradient}
\begin{array}{ll@{}ll}
||\nabla f(x) - \nabla f(y)|| \leq L_f ||x-y|| \quad \forall x,y \in \R^n.
\end{array}
\end{equation}
\end{definition}
If $f$ is $L_1$-smooth for some $L_1>0$, then it is also $L_2$-smooth for any $L_2>L_1$. Thus, we will often use $L_f$ to refer to the smallest Lipschitz constant.

\begin{definition}(Kernel generating distance \cite{bregman_descent_phase_retrieval})
Let $C$ be a nonempty, convex, and open subset of $\R^n$. A function $h: \R^n \to (-\infty,\infty]$ associated with $C$ is called a kernel generating distance if it satisfies the following
\begin{enumerate}
    \item $h$ is proper, lower semi-continuous, and convex, with $\dom(h) \subset \Bar{C}$ and \\ $\dom( \partial h) = C$.
    \item h is continuously differentiable in the interior of its domain.
\end{enumerate}
\end{definition}

Given a kernel generating distance $h$, the Bregman distance generated by the function $h$, denoted with $D_h: \dom(h)\times \interior \dom(h) \to \R_+ $ is defined as \cite{bregman_original}

\begin{equation}
\label{eq:bregman_distance}
\begin{array}{ll@{}ll}
D_h(y,x) = h(y) - h(x) - \nabla h(x)^T(y-x).
\end{array}
\end{equation}

The Bregman distance is like a norm, it is non-negative by gradient inequality for convex functions, and $D_h(y,x) = 0$ if and only if $y=x$ if we further assume that $h$ is strictly convex. However, unless $h$ is specially chosen, like the $\ell_2$ norm, Bregman distance does not satisfy the symmetry property. Bregman distance has been studied extensively in literature, see e.g., \cite{bregman_descent}.

\begin{definition}[\Ls-smooth adaptability \cite{bregman_descent}] Let $f,h$ be two continuously differentiable convex functions, $f$ is \Ls-smooth relative to $h$ if $Lh-f$ is convex for some $L> 0$:
\end{definition}

We will use a convention similar to the Lipschitz smoothness and use $L_h$ to denote the smallest Bregman constant. 

The following two descent lemmas characterize the smoothness of $f$ with respect to $h$ in terms of upper bounding $f$ with the Bregman distance generated by $h$

\begin{lemma}[Descent Lemma \cite{first_order_amir_beck}] $f$ is $\Ls$ smooth if and only if
\begin{equation}
\label{eq:lipschitz_descent}
\begin{array}{ll@{}ll}
f(y) \leq f(x) + \nabla f(x)^T(y-x) + \frac{L}{2}||y-x||_2^2.
\end{array}
\end{equation}
\end{lemma}

\begin{lemma}[Bregman or Extended Descent Lemma \cite{bregman_descent}] $f$ is \Ls-smooth relative to $h$ for some $L$ if and only if 
\begin{equation}
\label{eq:bregman_descent_lemma}
\begin{array}{ll@{}ll}
f(y) \leq f(x) + \nabla f(x)^T(y-x) + LD_h(y,x).
\end{array}
\end{equation}

\end{lemma}

The proofs of these characterizations follow from the gradient inequality for convex functions. 

\section{Bregman Stationarity}
\label{sec:bregman_stationarity}
Beck and Eldar proved the following result on \Ls-stationarity \cite{amir_beck_sparse_opt}
where $\mathcal{P}_{C_s}$ denotes the projection operator, $\mathcal{P}_{C_s}(y) = \underset{z \in C_s}\argmin ||z - y||_2^2$.
\begin{theorem}[\Ls-Stationarity]
\label{theorem:l_stationarity}
Let $x^*$ be an optimal solution to (\ref{eq:optimization_model}, where $f$ is $\Ls_f$-smooth. Then it satisfies
\begin{equation}
\label{eq:L_stationarity}
\begin{array}{ll@{}ll}
x^* \in \mathcal{P}_{C_s}(x^* - \frac{1}{\Ls_f}\nabla f(x^*)).
\end{array}
\end{equation}
\end{theorem}
  Set inclusion can be sharpened to equality if instead of $\Ls_f$, any $L > \Ls_f$ is used. 

\begin{proof}
Rearranging terms in the descent lemma and completing the squares we get
\begin{equation}
\label{eq:lipschitz_descent_lemma}
\begin{array}{ll@{}ll}
f(y) + \frac{L_f-L}{2}||y-x^*||_2^2 \leq f(x^*) + \frac{L}{2}||y-(x^*-\frac{1}{\Ls_f} \nabla f(x^*))||_2^2.
\end{array}
\end{equation}
Optimizing the RHS over the sparse set of vectors $y \in C_s$, we get $y^* = \mathcal{P}_{C_s}(x^* - \frac{1}{\Ls_f}\nabla f(x^*))$. Since $x^*$ is also $\in C_s$, by plugging in the optimal vector in the original equation we get
\begin{equation}
\label{eq:lipschitz_descent_lemma2}
\begin{array}{ll@{}ll}
f(y^*) + \frac{L_f-L}{2}||y^*-x^*||_2^2 \leq f(x^*).
\end{array}
\end{equation}
This means $y^*$ has a smaller or equal objective than the current point. In the case of $L > L_f$, $f(y^*)$ strictly decreases if $y \neq x$.
\end{proof}
This is a useful result that justifies all gradient-based algorithms in the literature like Iterative Hard Thresholding (IHT) and many others, c.f., \cite{amir_beck_sparse_opt,blumsignal,blumdaviesharmonic,blumdavies-selected,Blumensath-davies,gpnp,nhtp}.

A direct observation from the above proof is that the use of $\ell_2$ norm as a Bregman function is restrictive. Another intuition is that assuming $f$ is twice continuously differentiable, we use a multiple of identity to majorize Hessian to find a quadratic upper bound on $f$. Consider a generalization of this first-order condition using the Bregman descent lemma \cite{bregman_descent}.

\begin{lemma}[Bregman Stationarity]
\label{theorem:bregman_stationarity}
Assume $f$ is $L_h$-smooth relative to strictly convex function $h$. Let $x^*$ be an optimal solution to $P$. Let $L > L_h$ be a given constant. Then $x^*$ satisfies the following ;
\begin{equation}
\label{eq:bregman_stationarity}
\begin{array}{ll@{}ll}
x^* \in \underset{y \in C_s}\argmin \ \nabla f(x^*)^T(y-x) + LD_h(y,x^*).
\end{array}
\end{equation}

\end{lemma}

\begin{proof}
The proof follows similarly to the $L$-stationarity. Using the Bregman Descent lemma and adding additional terms of $\frac{L-L_h}{2}D_h(y,x)$ to both sides of the inequality gives
\begin{equation}
\label{eq:bregman_stationarity2}
\begin{array}{ll@{}ll}
f(y)+ \frac{L-L_h}{2}D_h(y,x) &\leq f(x) + \nabla f(x)^T(y-x) + L_hD_h(y,x). \\
&\leq f(x)
\end{array}
\end{equation}
Thus, if there exists $y \in C_s$ different from $x^*$ such that it minimizes the objective in (\ref{eq:bregman_stationarity}), $x^*$ is not an optimal point. As before, using the fact that $h$ is strictly convex, and therefore $D_h(y,x) = 0$ if and only if $y=x$, the minimizer is unique.
\end{proof}

However, this result is hardly useful in practice. It is difficult to solve this minimization problem implied by Bregman stationarity without assuming some structure about the $h$ function. Consider the Bregman function $h(x) = ||Ax||_2^2$. The  following quantities are needed in the development that follows:
\begin{equation}
\label{eq:H_approximation_LS}
\begin{array}{ll@{}ll}
f(x) = ||Ax-b||_2^2 = x^TA^TAx - 2b^TAx + b^Tb \\
\nabla f(x) = 2A^T(Ax-b), D_h(y,x) = ||A(y-x)||_2^2\\
g(y,x) \overset{\Delta}{=} f(x) + \nabla f(x)^T(x-y)  + (y- x^*)^T A^TA (y - x^*). \\
\end{array}
\end{equation}

In the last equation above, we see that the minimization of the upper bound on $f$ with the Bregman function $h$ is equivalent to solving a quadratic program with a sparsity constraint. If we could solve this subproblem then we could also solve sparsity constrained least squares problem. However, it is well known that the sparsity-constrained least squares problem is NP-Hard \cite{sparse_ls_np_hard}. 

The main drawback of arbitrary choice of $h$ is caused by variables being coupled, which complicates the problem because of the combinatorial nature of the sparsity constraint. We were able to minimize efficiently the upper bound on $f$ when it was $\Ls$-smooth, meaning the choice of the Bregman function was squared $\ell_2$ norm. The reason is that the problem is in fact separable when we use the Bregman function as squared $\ell_2$ norm. Additional simplicity of the function further allows us to compute the descent direction directly. To see this, consider a separable Bregman function of the form $h(x) = \sum_{i=1}^n h_i(x_i)$ and write the upper bound on $f$ as
\begin{equation}
\label{eq:bregman_descent_seperable}
\begin{array}{ll@{}ll}
D_h(y,x) = \sum_{i=1}^n h_i(y_i) - \sum_{i=1}^n h_i(x_i) - \sum_{i=1}^n h_i^\prime(x_i) (y_i-x_i),  \\
f(y) \leq f(x) + (\nabla f(x) - \nabla h(x))^T(y-x) + \sum_{i=1}^n h_i(y_i) - \sum_{i=1}^n h_i(x_i) = g_h(y,x),
\end{array}
\end{equation}
where $h_i^\prime$ denotes the derivative of the univariate function $h_i$. For this choice of $h$, the minimizer of $g_h(y,x)$ can be computed with a simple algorithm: 
\begin{equation}
\label{eq:bregman_descent_seperable_solution}
\begin{array}{ll@{}ll}
&u_i = \underset{z_i} \argmin \ \phi_i(z_i) = \nabla_i f(x)(z_i-x_i) + LD_{h_i}(z_i,x_i) \quad \forall i = 1, .. , n \\
&y_{j_k} =  u_{j_k} \quad \forall k = 1, .. , s \\
&y_{j_k} =  0      \qquad \forall k = s+1, .., n \\
& \phi_{j_1}(0) - \phi_{j_1}(u_{j_1}) \geq \phi_{j_2}(0) - \phi_{j_2}(u_{j_2}) \geq ....
\end{array}
\end{equation}
The algorithm described in (\ref{eq:bregman_descent_seperable_solution}) reduces to the following: for each variable, solve the one-dimensional optimization problem over $\phi_i$ and compute $u_i$. Then, compare the minimum value of the $\phi_i$ against the function evaluated at zero meaning comparison of $\phi_i(0) - \phi_i(u_i)$. Choose the largest $s$ indices that give the best descent value and set $y_i$'s to be the corresponding minimizers, $u_i$'s for those indices.

This idea motivates us to find separable Bregman functions to adapt to $f$. In this paper, we focus on weighted $\ell_2$ norm type functions as a special case. In other words, we parameterize separable squared loss with weights, $h(x) = \frac{1}{2}||x||_D^2 = \frac{1}{2}x^TDx$. In Section \ref{sec:numerical_results} we show an application of our methods on least squares type problems, where the function $f$ is twice differentiable and has constant Hessian. These properties of $f$ allow us to compute the weights for the Bregman function easily, as we show in Section \ref{sec:compute_weights}. However, first, we show convergence results in the abstract setting and prove a new result.

We note here that beyond the $\ell_2$ smooth functions, the Bregman Proximal Gradient (BPG) was applied to Non-Negative Matrix Factorization (NMF) problem in \cite{bpg_alternating_nmf,block_bregman_bpg_nmf,bi_smooth_nmf} where the implicit sparsity constraint is imposed via rank constraint. In \cite{nmf_bregman_teboulle_yakov}, the same problem with additional explicit sparsity constraint is formulated, and a BPG algorithm is developed. Additionally, a BPG algorithm for sparsity constrained Phase Retrieval problem is developed in \cite{bregman_descent_phase_retrieval}. In the above construction, we propose a simple generic procedure that allows for deriving gradient-based algorithms for sparsity-constrained problems.

\section{Algorithm and Convergence Analysis}
\label{sec:convergence}
In this section, we assume $f(x)$ is $L_h$-smooth relative to $h(x) = \frac{1}{2}x^THx$ for some diagonal matrix $H$.

Since this is a quadratic form, we will also refer to this property as $f(x)$ being $H$-smooth. Therefore there exists $L_h > 0$ such that
\begin{equation}
\label{eq:H_stationary}
\begin{array}{ll@{}ll}
f(y) \leq f(x) + \nabla f(x)^T(y-x) + \frac{L_h}{2}||y-x||_H^2,
\end{array}
\end{equation}
where $||z||_H^2 = z^THz$. For simplicity, we will assume that $L_h = 1$, since this can be achieved by using another Bregman function $\hat{h}(x) = \frac{1}{2}x^T\hat{H}x$, with $\hat{H} = L_hH$. From the result in Lemma \ref{theorem:bregman_stationarity} we see that to guarantee the uniqueness of the minimizer of projection in the descent lemma, we need a step size $L$ such that $L > L_h$.

Thus, we will assume that we are working with a diagonal matrix $D$ such that $D \succeq \nabla ^2 f + G$ holds for some $G \succ 0$, which can be achieved by either computing a diagonal matrix $H$ that satisfies (\ref{eq:H_stationary}) 
and perturbing it with a small multiple of identity matrix or multiplying it with a number strictly larger than one.

Firstly, instead of using the algorithm template in (\ref{eq:bregman_descent_seperable_solution}), the solution can be computed using a much simpler operation. Starting by explicitly writing (\ref{eq:bregman_descent_lemma})
\begin{equation}
\label{eq:D_descent}
\begin{array}{ll@{}ll}
f(y) \leq f(x) +  \nabla f(x)^T(y-x) + \frac{1}{2}(y-x)^TD(y-x) \\
y^* = \underset{y\in C_s}\argmin \nabla f(x)^T(y-x) + \frac{1}{2}(y-x)^TD(y-x) \\
= \underset{y \in C_s}\argmin \, ||D^{1/2}(y-x) - D^{-1/2}\nabla f(x) ||_2^2 \\
= \underset{z,y \in C_s,z = D^{1/2}y}\argmin \, ||z-D^{1/2}x - D^{-1/2}\nabla f(x) ||_2^2. \\
\end{array}
\end{equation}
Going from the second to the third equation we ignored constant terms to simplify the objective. Then we introduce an auxiliary variable $z$ which is an elementwise scaled version of $y$ so they have the same sparsity pattern, i.e. $I_1(y) = I_1(z)$. The last equation is simply a projection to the $s$ sparse set of vectors where we can extract the solution as
\begin{equation}
\label{eq:D_descent_solution}
\begin{array}{ll@{}ll}
z^* &\in {\cal P}_{C_s}\left(D^{1/2} x - D^{-1/2} \nabla f(x)\right) \\
y^* &= D^{-1/2}z. 
\end{array}
\end{equation}
Although non-convex, projection to $s$-sparse vectors can be easily computed by truncating the vector by its largest $s$ values in magnitude.
\begin{equation}
\label{eq:projection_solution}
\begin{array}{ll@{}ll}
{\cal P}_{C_s}\left(x\right) = \underset{y\in C_s}\argmin ||y-x||_2^2.
\end{array}
\end{equation}
\begin{definition}
Let $x \in C_s$ be a sparse vector. Then $x$ is called a $D$-stationary point if it satisfies the following:
\begin{equation}
\label{eq:D_stationary_point}
\begin{array}{ll@{}ll}
x \in D^{-1/2} {\cal P}_{C_s}\left(D^{1/2} x - D^{-1/2} \nabla f(x)\right).
\end{array}
\end{equation}
\end{definition}
There is an alternative characterization of a $D$-stationary point in terms of the magnitude of the gradient.
\begin{lemma}
\label{lemma:D_stationary_point_characterization}
Given a diagonal matrix $D \succ 0$, a point $x$ is $D$-stationary if and only if 
\begin{equation}
\label{eq:D_stationary_point_characterization}
\begin{array}{ll@{}ll}
|\nabla_i f(x)|
\begin{cases} 
  = 0 &\textit{if} \ i \in I_1(x) \\
  \leq D_{ii}^{1/2}M_s(D^{1/2}x) &\textit{if} \ i \in I_0(x). \\
\end{cases}
\end{array}
\end{equation}
\end{lemma}

\begin{proof}
$\implies$ If $x$ is $D$-stationary, then $D^{1/2}x \in {\cal P}_{C_s}\left(D^{1/2} x - D^{-1/2} \nabla f(x)\right)$. For $i \in I_1(x)$, since $x_i$ is positive, it must be recovered from the projection $D_{ii}^{1/2}x_i = D_{ii}^{1/2}x_i - D_{ii}^{-1/2} \nabla_i f(x)$, which implies $\nabla_i f(x) = 0$. For $i \in I_0(x)$, $x_i$ is being set to $0$ by either $||x||_0 < s$ and both $x_i$ and $\nabla_i f(x)$ are zero or $||x||_0 = s$ and hence projection selects $s$ non-zero values and does not select $x_i$ since it is not one of the largest $s$ values in $D^{1/2} x - D^{-1/2} \nabla f(x)$. Then, it is smaller than (or equal to) the largest $s$'th value in magnitude in the projected vector, which implies $|D_{ii}^{1/2}x_i - D_{ii}^{-1/2} \nabla_i f(x)| \leq M_s(D^{1/2}x)$. Also using $x_i =0$ since $i \in I_0(x)$, we get $|\nabla_i f(x)| \leq D_{ii}^{1/2}M_s(D^{1/2}x)$.

$\impliedby$ If $||x||_0 < s$ then $|\nabla_i f(x)| \leq D_{ii}^{1/2}M_s(D^{1/2}x) \ \forall i \in I_0(x)$ and $\nabla_i f(x) = 0 \ \forall i \in I_1(x)$ implies $||\nabla f(x)||_2 = 0$, so gradient is zero for all variables, and thus gradient projection can be ignored. This in turn implies, $D^{-1/2} {\cal P}_{C_s}\left(D^{1/2} x \right) = x$ and $x$ is $D$-stationary.

If $||x||_0 = s$, since $\nabla_i f(x) = 0$ for $i \in I_1(x)$, and $x_j = 0$ for $j \in I_0(x)$, term inside the projection becomes either $D_{ii}^{1/2}x_i$ or $D_{jj}^{-1/2}\nabla_j f(x)$ for a given index. Since $\nabla_j f(x) \leq D_{jj}^{1/2}M_s(D^{1/2}x)$, the largest $s$ values of $D^{1/2} x - D^{-1/2} \nabla f(x)$ are the non-zero $x$ values in $D^{1/2}x$ (with possibly tie), showing that $x$ is $D$-stationary. 
\end{proof}

We note that if one has $x$, a $D$-stationary point with $||x||_0 < s$ this means that $||\nabla f(x)||_2 = 0$, which indicates a global minimizer of the function ignoring sparsity constraint since $f$ is convex.  

We now present the IWHT Algorithm \ref{alg:iwht}. 

\begin{algorithm}[ht]
\caption{Iterative Weighted Hard Thresholding (IWHT)}
\label{alg:iwht}
\begin{algorithmic}[1]
\STATE{\textbf{Input:} $D, x_0 \in C_s $}
\FOR{$k = 1,2, ...,$}
\STATE{$y_{k+1} = \mathcal{P}_{C_s}(D^{1/2}x_k-D^{-\frac{1}{2}}\nabla f(x_k))$}
\STATE{$x_{k+1} = D^{-1/2}y_{k+1}$}
\ENDFOR
\end{algorithmic}
\end{algorithm}

Well-known convergence results related to iterative hard thresholding (IHT) algorithm similarly follow for IWHT as well. 

\subsection{Subsequential Convergence}
\begin{lemma}
\label{lemma:subsequential_convergence}
Let $(x_k)$ be a sequence generated by the IWHT method with Bregman function $h(x) = \frac{1}{2}x^TDx$ for the function $f(x)$ which is $H$-Smooth with $D \succ H$. Then the following holds:
\begin{enumerate}
    \item $f(x_{k+1}) \leq  f(x_k) - \frac{1}{2}||x_{k+1}-x_k||_{D-H}^2$. 
    \item Function sequence $f(x_k)$ is nonincreasing.
    \item $\sum_{k=0}^{\infty} ||x_{k+1}-x_k||_{D-H}^2 < \infty$, and hence $\lim_{k\to \infty} ||x_{k+1}-x_k||_{D-H}^2 = 0$
    \item Minimum stationarity gap $\underset{k\leq n}\min ||x_{k+1}-x_k||_{D-H}^2$ decreases with rate $O(\frac{1}{n})$.
\end{enumerate}
\end{lemma}

\begin{proof}
Assuming $f$ is bounded below, statements 2, 3 and 4 are direct consequences of the first statement. To see the correctness of the first statement, adding a $\frac{1}{2}||x-y||_{D-H}^2$ term to both sides of the descent lemma with $H$ matrix (\ref{eq:H_stationary}) with $y = x_{k+1}$ and $x = x_k$ gives
\begin{equation}
\label{eq:iwht_descent_lemma}
\begin{array}{ll@{}ll}
f(x_{k+1}) + \frac{1}{2}||x_{k+1}-x_k||_{D-H}^2 &\leq f(x_k) + \nabla f(x_k)^T(x_{k+1}-x_k) + \frac{1}{2}||x_{k+1}-x_k||_D^2 \\ 
&\leq f(x_k).
\end{array}
\end{equation}
In the second line, we used the optimality of $x_{k+1}$, as the minimizer of the RHS of the first line. 
\end{proof}

These results are a slight generalization of \cite{amir_beck_sparse_opt} while more general results are analyzed in \cite{bregman_descent_phase_retrieval,bregman_descent,relatively_smooth_opt}, which follow from the same analysis using the Extended Descent Lemma with $Lh-f$ when $L > L_h$. 

Although it is not clear whether the $(x_k)$ sequence converges, the sequence might have at least some limit points. The next result shows that any limit point $\Bar{x}$ of $(x_k)$ satisfies $D$-stationarity. 

\begin{lemma}
\label{lemma:limit_d_stationary}
In the setting of Lemma \ref{lemma:subsequential_convergence}, any limit point of $(x_k)$ is $D$-stationary. 
\end{lemma}

\begin{proof}
Assume to the contrary, a limit point $\Bar{x}$ is not $D$-stationary. Then $\exists y \in C_s$ such that
\begin{equation}
\label{eq:iwht_L_stationary}
\begin{array}{ll@{}ll}
\nabla f(\Bar{x})^T(y-\Bar{x}) + \frac{1}{2}(y-\Bar{x})^TD(y-\Bar{x}) = - \delta < 0. 
\end{array}
\end{equation}
Consider a subsequence $(x_{n_k})$ that converges to $\Bar{x}$. When $(x_{n_k})$ is close enough to $\Bar{x}$, the previous inequality will be preserved. Then using the descent lemma with the previous equation for large enough $k$ we have
\begin{equation}
\label{eq:iwht_L_stationary2}
\begin{array}{ll@{}ll}
f(x_{n_k+1}) \leq f(x_{n_k}) +  \nabla f(x_{n_k})^T(x_{n_k+1}-x_{n_k}) + \frac{1}{2}(x_{n_k+1}-x_{n_k})^TH(x_{n_k+1}-x_{n_k}) \\
f(x_{n_k+1}) \leq f(x_{n_k}) - \frac{1}{2}||x_{n_k+1}-x_{n_k}||_{D-H}^2 - \delta/2.
\end{array}
\end{equation}
This shows that around $\Bar{x}$, after applying the gradient iteration, the objective decreases at least by some positive amount $\delta/2$. However, this result contradicts the objective value convergence from Lemma \ref{lemma:subsequential_convergence}.
\end{proof}


Since our goal was to identify new descent conditions beyond the gradient descent with $\Ls$-stationarity, we could use multiple $D_i$'s say in a cyclic form. Let $\mathcal{D} = \cup_i \{D_i\}$ denote the set of matrices to apply the proximal gradient descent. This idea gives rise to the Cyclic IWHT Algorithm \ref{alg:ciwht}.

\begin{algorithm}[ht]
\caption{Cyclic Iterative Weighted Hard Thresholding (CIWHT)}
\label{alg:ciwht}
\begin{algorithmic}[1]
\STATE{\textbf{Input:} $\mathcal{D}, x_{0,0} \in C_s $}
\FOR{$k = 1,2, ...,$}
    \STATE{$x_{k,0} = x_{k-1,m}$}
        \FOR{$m = 1,2,.., |\mathcal{D}|$}
        \STATE{$y_{k,m} = \mathcal{P}_{C_s}(D_{m}^{\frac{1}{2}}x_{k,m-1}-D_{m}^{-\frac{1}{2}}\nabla f(x_{k,m-1}))$}
        \STATE{$x_{k,m} = D_{m}^{-\frac{1}{2}}y_{k,m}$}
        \ENDFOR
\ENDFOR
\end{algorithmic}
\end{algorithm}

\begin{theorem}
\label{lemma:d_stationary_cyclic}
Assume $f$ is $H_i$-stationary and $D_i \succ H_i$ for $i = 1, .., |\mathcal{D}|$. Any limit point of $\Bar{x}$ the sequence generated by the CIWHT algorithm $(x_{|\mathcal{D}|k + m}) = (x_{k,m}) $ for $ k \geq 1, 1 \leq m \leq |\mathcal{D}|$ is stationary with respect to all Bregman functions.
\end{theorem}
\begin{proof}
First of all, results of Lemma \ref{lemma:subsequential_convergence} apply in this context similarly by using $D_i \succ H_i$. Thus the distance of consecutive points converges to zero. The rest of the proof of this theorem is identical to Lemma \ref{lemma:limit_d_stationary}. We can find a convergent subsequence to some limit point $\Bar{x}$ and make the sequence sufficiently close to it so that after completing the cycle of inner $|\mathcal{D}|$ iterations, by also combining $||x_{k+1}-x_k||_{D_i-H_i} \to 0 \ \forall \, i = 1, ... , \mathcal{D}$ or $\sum_{i=1}^{\mathcal{D}}||x_{k+i}-x_{k+i-1}||_{D_i-H_i} \to 0$, the descent criterion in equation (\ref{eq:iwht_L_stationary}) is still strictly satisfied. Then, this creates a contradiction to objective value convergence as before.
\end{proof}

\subsection{Convergence of the Point Sequence}

Although we have subsequential convergence results from Lemma \ref{lemma:subsequential_convergence} and a result on the property of the limit points from Lemma \ref{lemma:limit_d_stationary}, convergence is not guaranteed for the whole sequence. Often a strong convexity-like assumption introduced in \cite{agarwal2010fast,bahmani2013greedy,Blumensath2012CompressedSW,high_dimensional_decomposable_regularizer,shalev2010trading,yuan_grad_hard_threshold} (for convergence analysis in a different context) is made to prove that the whole sequence converges \cite{amir_beck_sparse_opt,gpnp,nhtp}.

\begin{definition}
Function $f(x)$ is called s-Restricted Strongly Convex (RSC) if there exist $\sigma_s > 0$ such that
\begin{equation}
\label{eq:restricted_strong_convexity}
\begin{array}{ll@{}ll}
f(y) \geq f(x) + \nabla f(x)^T(y-x) + \frac{\sigma_s}{2}||x-y||_2^2 \quad  \forall x,y: ||x-y||_0 \leq s.
\end{array}
\end{equation}
\end{definition}
Intuitively, the s-RSC property means the function behaves like a strongly convex function between two points if they have the same sparsity pattern. We now present one version of proof of convergence of the $(x_k)$ sequence under restricted strong convexity assumption.

\begin{proposition}
\label{lemma:rsc_convergence}
In the setting of Lemma \ref{lemma:subsequential_convergence}, if additionally, $f$ satisfies s-RSC, then the entire sequence $(x_k)$ converges.
\end{proposition}

\begin{proof}
From Lemma \ref{lemma:subsequential_convergence}, and s-RSC, the sequence $(x_k)$ is bounded and hence there exists a limit point $\Bar{x}$ that satisfies $D$-stationarity from Lemma \ref{lemma:limit_d_stationary}. For simplicity assume $||\Bar{x}||_0 = s$. To the contrary, assume that the entire sequence $(x_n)$ does not converge to $\Bar{x}$. Since $||x_k-x_{k+1}||_{D-H}^2 \to 0$, for large enough $k$, sparsity pattern of the sequence $(x_k)$ will not change immediately. Further, there exist a subsequence $(x_{n_k})$ that converges to $\Bar{x}$. Choose $k$ large enough, say $k \geq K$, so that both $f(x_{n_k})-f(\Bar{x}) < \epsilon_1$ and $||x_k-x_{k+1}||_{D-H}^2 < \epsilon_2$ are satisfied. Let $m$ be the first iteration after $K$ where the sparsity pattern changes. Then, writing the restricted strong convexity equation with respect $x_{m-1}$
\begin{equation}
\label{eq:rsc_convergence}
\begin{array}{ll@{}ll}
f(x_{m-1}) \geq f(\Bar{x}) + \nabla f(\Bar{x})^T(x_{m-1}-\Bar{x}) + \sigma_s ||x_{m-1}-\Bar{x}||_2^2 \\
f(x_{m-1}) \geq f(\Bar{x}) + \sigma_s ||x_{m-1}-\Bar{x}||_2^2
\end{array}
\end{equation}
In the second equation, we used the fact that they have the same sparsity pattern, therefore, both $x_{m-1}$ and $\Bar{x}$ are zero outside the same indices. Furthermore, wherever $\Bar{x}$ has a non-zero value, the corresponding gradient is zero by $D$-stationarity. Moreover, since in the next iteration sparsity pattern changes, and $||x_m-x_{m-1}||_{D-H}^2 < \epsilon_1$, points $\Bar{x}$ and $x_{m}$ are separated by at least some distance, i.e $||x_{m-1}-\Bar{x}||_{D-H}^2 \geq \delta > 0 $. This creates a contradiction to $f(x_{n_k})-f(\Bar{x}) < \epsilon_1$. Similarly, if the sparsity pattern does not change, points in the sequence $(x_k)$ cannot stay outside of some ball infinitely often with the same reasoning. Thus the sequence $(x_k)$ converges to $\Bar{x}$.


If $||\Bar{x}|| < s$ this case can be argued similarly since $(x_k)$ sequence is kept $s$-sparse throughout iterations, we can write the same equation by also using $\nabla f(\Bar{x}) = 0$.
\end{proof}

In \cite{bregman_descent_phase_retrieval}, it was shown that if  $f$ satisfies Kurdyka-\L{}ojasiewicz (K\L) property then the sequence $(x_n)$  also converges for gradient descent like sequences. The growth condition given in \cite{proximal_alternating_linearized} shows that uniformly convex or in particular strongly convex functions satisfy this property. These results could also be used to assert convergence of the IWHT algorithm. However, we note that the assumption of $s$-RSC used in Proposition \ref{lemma:rsc_convergence} is a weaker notion. We further extend this result to our setting without using KL property or restricted strong convexity criterion.


\begin{theorem}
\label{thm:L_convergence}
In the setting of Lemma \ref{lemma:subsequential_convergence}, assume the sequence $(x_k)$ has a limit point $\Bar{x}$. Then the entire sequence converges to $\Bar{x}$.
\end{theorem}
\begin{proof}
Consider the case when $||\Bar{x}||_0 = s$. Consider a subsequence $(x_{n_k})$ that converges to $\Bar{x}$. As in the proof of the previous theorem, using Lemma \ref{lemma:subsequential_convergence}, for large enough $k$ we have $||x_k-x_{k+1}||_{D-H}^2 < \epsilon_1$ and $||\Bar{x}-x_{n_k}|| < \epsilon_2$. Choosing $\epsilon_1$ and $\epsilon_2$ sufficiently small, the next iterate $(x_{n_k+1})$ in the IWHT algorithm will have the same sparsity pattern as $x_{n_k}$ and hence $\Bar{x}$, i.e $I_1(x_{n_k+1}) = I_1(x_{n_k}) = I_1(\Bar{x}) = S$. Then, define the following subvectors which are all in the same space $\R^s$;
\begin{equation}
\label{eq:D_convergence1}
\begin{array}{ll@{}ll}
z_{n_k} = (x_{n_k})_{S} \\
z_{n_k+1} = (x_{n_k+1})_{S} \\
\Bar{z} = (\Bar{x})_{S} \\ 
\Bar{f}(z) = f(x), \quad x_S = z, x_{S^C} = 0 \\
\nabla \Bar{f}(z) = \nabla_S f(x),
\end{array}
\end{equation}
where we also defined a restricted function $\Bar{f}$, which takes an input in $\R^s$ and is equal to function evaluated at $f(x)$ where $x$ is zero-padded version of $z$ with the sparsity pattern $S$. Clearly, the restricted function $\Bar{f}$ is $D_S$ smooth. Then analyzing the next iterate in the restricted space gives;
\begin{equation}
\label{eq:D_convergence2}
\begin{array}{rlll}
x_{n_k+1} &= D^{-1/2}\mathcal{P}_{C_s}(D^{1/2}x_{n_k}-D^{-\frac{1}{2}}\nabla f(x_{n_k})) \\
\implies z_{n_k+1} &= z_{n_k} - D_S^{-1}\nabla \Bar{f}(z_{n_k}).
\end{array}
\end{equation}
Therefore, $z_{n_k+1}$ reduces to a gradient descent iteration with weighted $\ell_2$ norm Lipschitz smoothness criterion. It is well known that in a gradient descent sequence, the distance to the optimal point decreases when the gradient is smooth,
\begin{equation}
\label{eq:D_convergence3}
\begin{array}{llll}
||z_{n_k+1}-\Bar{z}||_{D_S}^2 &= ||z_{n_k} - D_S^{-1}\nabla \Bar{f}(z_{n_k}) - \Bar{z}||_{D_S}^2 \\
& = ||z_{n_k} - \Bar{z}||_{D_S}^2 - 2 \nabla \Bar{f}(z_{n_k})^T(z_{n_k} - \Bar{z}) + ||D_S^{-1}\nabla \Bar{f}(z_{n_k})||_{D_S}^2 \\
& \leq ||z_{n_k} - \Bar{z}||_{D_S}^2 - 2 ||\nabla\Bar{f}(z_{n_k}) ||_{D_S^{-1}}^2 + ||\nabla \Bar{f}(z_{n_k})||_{D_S^{-1}}^2 \\
& \leq ||z_{n_k} - \Bar{z}||_{D_S}^2 - ||\nabla\Bar{f}(z_{n_k}) ||_{D_S^{-1}}^2.
\end{array}
\end{equation}
In the second line, we used the fact that Lipschitz continuity of   $\Bar{f}$ implies that $\nabla \Bar{f}(z_{n_k})^T(z_{n_k} - \Bar{z}) \geq || \nabla \Bar{f} ||_{D_S^{-1}}$ \cite{first_order_amir_beck}. In the third line, we used $||A^{-1}x||_A = ||x||_{A^{-1}}$. Then, in the next iteration, $z_{n_k+1}$ is even closer to $\Bar{z}$ and similarly for $x_{n_k+1}$ and $\Bar{x}$ since they are non-zero on the same sparsity pattern and zero elsewhere. Then, for the rest of the iterations, IWHT will never change the sparsity pattern and will keep doing gradient descent with respect to $D$ norm in the restricted space i.e. the identified sparsity pattern. With the same argument, from the convergence of gradient descent, it is evident that the sequence $(x_k)$ converges to $(\Bar{x})$.

If $\Bar{x} < s$, by Lemmas \ref{lemma:limit_d_stationary} and \ref{lemma:D_stationary_point_characterization} we have that $\nabla f(\Bar{x}) = 0$. Firstly, since the algorithm will keep $s$-sparse iterates, and $x_{n_k}$ close enough to $\Bar{x}$ we have $I_1(x_{n_k}) \supset I_1(\Bar{x})$. Then two things can happen after an IWHT iteration when $x_{n_k}$ is sufficiently close to $\Bar{x}$. If $I_1(x_{n_k}) = I_1(x_{n_k+1})$, with the same gradient argument above $x_{n_k+1}$ gets even closer to $\Bar{x}$. If $I_1(x_{n_k}) \neq I_1(x_{n_k+1})$, we must still have that $I_1(x_{n_k+1}) \supset I_1(\Bar{x})$ from the third result in Lemma \ref{lemma:subsequential_convergence}. Then the distance to the optimal solution has the following relationship;
\begin{equation}
\label{eq:D_convergence4}
\begin{array}{rlll}
||
\begin{bmatrix}
z^+  \\
0 \\
-D_{S_3}^{-1}\nabla_{S_3} f(z)  \\
\end{bmatrix}
-
\begin{bmatrix}
\Bar{z} \\
0 \\
0  \\
\end{bmatrix}
||_{D_S}^2
= 
&||
\begin{bmatrix}
z  \\
u \\
0  \\
\end{bmatrix}
-
\begin{bmatrix}
D_{S_1}^{-1}\nabla_{S_1} \Bar{f}(z) \\
u \\
D_{S_3}^{-1}\nabla_{S_3} \Bar{f}(z)  \\
\end{bmatrix}
-
\begin{bmatrix}
\Bar{z} \\
0 \\
0  \\
\end{bmatrix}
||_{D_S}^2 \\
= & ||
\begin{bmatrix}
z  \\
u \\
0  \\
\end{bmatrix}
-
\begin{bmatrix}
\Bar{z} \\
0 \\
0  \\
\end{bmatrix}
||_{D_S}^2
- 2 \nabla_{S} \Bar{f}(z) ^T\begin{bmatrix}
z-\Bar{z} \\
u \\
0  \\
\end{bmatrix} \\ 
&-2 ||u||_{D_{S_2}}^2 + 2 \nabla_{S_2} \Bar{f}(z) ^Tu  \\
& + ||\nabla_{S_1} \Bar{f}(z) ||_{D_{S_1}^{-1}}^2 + ||u||_{D_{S_2}}^2 + ||\nabla_{S_3} \Bar{f}(z)||_{D_{S_3}^{-1}}^2

\end{array}
\end{equation}
where $S_1 = I_1(x_{n_k}) \cap I_1(x_{n_k+1}),S_2 = I_1(x_{n_k}) \setminus I_1(x_{n_k+1}), S_3 = I_1(x_{n_k+1}) \setminus I_1(x_{n_k})$, the sparsity patterns of the vectors and $S = S_1 \cup S_2 \cup S_3$ is their union, $z^+ = (x_{n_k+1})_{S_1}, z = (x_{n_k})_{S_1}, \Bar{z} =  (\Bar{x})_{S_1}$ are the vectors restricted to respective sparsity patterns and $ v = (x_{n_k+1})_{S_3} = -D_{S_3}^{-1}\nabla_{S_3} \Bar{f}(z), u = (x_{n_k})_{S_2}$ are the non-zero values of the vectors that are not in the common sparsity pattern. The restricted function $\Bar{f}$ is defined as zero-padded version of $x \in \R^{|S|}$.  We used the fact that $u = D_{S_2}^{-1}\nabla_{S_2} \Bar{f}(z) - D_{S_2}^{-1}\nabla_{S_2} \Bar{f}(z) +u$ and expanded the terms accordingly to get the inner product of the weighted gradient with $z-\Bar{z}$. Now using smoothness of the restricted function $\Bar{f}$ with respect to $D_S$ we get
\begin{equation}
\label{eq:D_convergence5}
\begin{array}{rlll}
2 \nabla_{S} \Bar{f}(z) ^T\begin{bmatrix}
z-\Bar{z} \\
u \\
0  \\
\end{bmatrix}
&\geq 2||\nabla_{S} \Bar{f}(z)||_{D_{S}^{-1}}^2 \\
&= 2||\nabla_{S_1} \Bar{f}(z) ||_{D_{S_1}^{-1}}^2 + 2||\nabla_{S_2} \Bar{f}(z) ||_{D_{S_2}^{-1}}^2 + 2||\nabla_{S_3}\Bar{f}(z) ||_{D_{S_3}^{-1}}^2
\end{array}
\end{equation}
This inequality in the restricted space of $S$ works because $\nabla f(\Bar{x}) = 0$. Plugging this inequality back gives
\begin{equation}
\label{eq:D_convergence6}
\begin{array}{rlll}
||
\begin{bmatrix}
z^+  \\
0 \\
-D_{S_3}^{-1}\nabla_{S_3} \Bar{f}(z)  \\
\end{bmatrix}
-
\begin{bmatrix}
\Bar{z} \\
0 \\
0  \\
\end{bmatrix}
||_{D_S}^2
\leq & 

||
\begin{bmatrix}
z  \\
u \\
0  \\
\end{bmatrix}
-
\begin{bmatrix}
\Bar{z} \\
0 \\
0  \\
\end{bmatrix}
||_{D_S}^2
- ||u||_{D_{S_2}}^2 + 2 \nabla_{S_2} \Bar{f}(z) ^Tu \\
-&||\nabla_{S_1} \Bar{f}(z) ||_{D_{S_1}^{-1}}^2 - 2||\nabla_{S_2}\Bar{f}(z)||_{D_{S_2}^{-1}}^2 \\
-&||\nabla_{S_3} \Bar{f}(z) ||_{D_{S_3}^{-1}}^2 \\
= & ||
\begin{bmatrix}
z  \\
u \\
0  \\
\end{bmatrix}
-
\begin{bmatrix}
\Bar{z} \\
0 \\
0  \\
\end{bmatrix}
||_{D_S}^2
- ||u-D_{S_2}^{-1}\nabla_{S_2} \Bar{f}(z)||_{D_{S_2}}^2 \\
-& ||\nabla_{S} \Bar{f}(z)||_{D_{S}^{-1}}^2
\end{array}
\end{equation}
In the second line we completed the squares $-||u||_{D_{S_2}}^2 + 2 \nabla_{S_2} \Bar{f}(z) ^Tu - ||\nabla_{S_2} \Bar{f}(z)||_{D_{S_2}^{-1}}^2 $. 

The last inequality implies that even when the sparsity pattern changes, the distance of the next iterate to the $\Bar{z}$ is smaller than the previous one. This argument will repeat itself in the next iteration without using the subsequence property, proving convergence of $x_k$ to $\Bar{x}$. 
\end{proof}

If $f$ satisfies any growth condition to imply that $(x_k)$ will remain in a bounded set then the convergence of the sequence is guaranteed. Merely assuming that $f(x)$ is coercive is enough for this purpose. Additionally, this can be sharpened as follows;
\begin{equation}
\label{eq:sparse_coercive}
\begin{array}{ll@{}ll}
\forall N > 0, \exists M_N > 0 \quad \text{s.t.} \\
||x|| > M_N \implies f(x_S)  > N \quad \forall ||x||_0 \leq s.
\end{array}
\end{equation}
This condition essentially asserts that $f$ is s-sparse coercive, i.e. for all sparse vectors, $f$ is coercive. Nevertheless, without assuming a limit point of the sequence or any growth condition on $f(x)$ convergence of the sequence $(x_k)$ is not ensured. The next result shows that this could only happen in a very artificial case. When it happens, the algorithm generates an approximately global optimum whenever it changes the sparsity pattern of the $(x_k)$. 

\begin{theorem}
\label{thm:sparsity_pattern_change}
In the setting of Lemma \ref{lemma:subsequential_convergence}, if the number of sparsity pattern changes in $(x_k)$ is finite, then the sequence $(x_k)$ converges to some optimal point in the last sparsity pattern. If the number of sparsity pattern changes is infinite, then the magnitude of the gradient at the points where the sparsity pattern changes goes to zero. 
\end{theorem}
\begin{proof}
The first part of the theorem follows from the gradient descent argument in the restricted space. Assume now that the number of sparsity pattern changes is infinite. Consider the subsequence of points where the sparsity pattern changes. Then  using Lemma \ref{lemma:subsequential_convergence} and plugging the candidate solution $x_{S} = x_{n_k}-D_{S}^{-1}\nabla_{S}f(x_{n_k}), x_{S^C} = 0$ where $S = I_1(x_{n_k})$, for the choice of $x_{n_k+1}$ in equation (\ref{eq:iwht_descent_lemma}) instead of $x_{n_k}$ gives
\begin{equation}
\label{eq:sparsity_pattern_change}
\begin{array}{llll}
f(x_{k+1}) + \frac{1}{2}||x_{n_k} - x_{n_k+1}||_{D-H} \leq  f(x_k) - \frac{1}{2}||\nabla_{S} f(x_{n_k})||_{D_{S}^{-1}} \\
\implies \frac{1}{2}||\nabla_{S} f(x_{n_k})||_{D_{S}^{-1}} + \frac{1}{2}||x_{n_k} - x_{n_k+1}||_{D-H} \leq f(x_k) - f(x_{k+1}) \\
\implies \frac{1}{2}||\nabla_{S} f(x_{n_k})||_{D_{S}^{-1}} + \frac{1}{2}||x_{n_k} - x_{n_k+1}||_{D-H} \leq \epsilon.
\end{array}
\end{equation}
In the third line we used the objective function value convergence result from Lemma \ref{lemma:subsequential_convergence}. This shows that at the iteration where the sparsity pattern changes, the norm of the gradient restricted to the sparsity pattern $S$ gets arbitrarily small. Since the sparsity pattern changed we can also bound the gradient norm in the complement of $S$ by discarding the first term on the left-hand side. For any $i \in I_1(x_{n_k+1}) \setminus I_1(x_{n_k})$, we have the following
\begin{equation}
\label{eq:sparsity_pattern_change2}
\begin{array}{rlll}
\dfrac{(\nabla_i f)^2}{d_i^2} (d_i-h_i) & \leq \epsilon \quad &\forall i \in I_1(x_{n_k+1}) \setminus I_1(x_{n_k}) \\
\dfrac{(\nabla_i f)^2}{d_i^2} &\leq C_1 \epsilon \quad  &\forall i \in I_1(x_{n_k+1}) \setminus I_1(x_{n_k}) \\
\dfrac{(\nabla_i f)^2}{d_i} &\leq C_1 C_2 \epsilon \quad &\forall i \in I_1(x_{n_k+1}) \setminus I_1(x_{n_k}) \\
\implies \underset{j\in I_0(x_{n_k})}\max \dfrac{(\nabla_j f)^2}{d_j}&\leq  C_1 C_2 \epsilon \\
\implies \underset{j\in I_0(x_{n_k})}\max (\nabla_j f)^2 &\leq C_1 C_2^2 \epsilon,
\end{array}
\end{equation}
where $C_1$ is a uniform lower bound on $d_i-h_i$ and $C_2$ is a uniform upper bound on $d_i$. In the first, second, and last lines, we use the lower and upper bounds to uniformize the bound. Going from the third to the fourth line, we use the fact that when the sparsity pattern changes, newly chosen sparsity patterns will be chosen with the maximum ratio of $\dfrac{(\nabla_i f)^2}{d_i}$ as given by the algorithm (\ref{eq:bregman_descent_seperable_solution}) or (\ref{eq:D_descent_solution}). This means that the norm of the gradient on $I_0(x_{n_k})$ becomes arbitrarily small. 

Combining both (\ref{eq:sparsity_pattern_change}) and (\ref{eq:sparsity_pattern_change2}) shows that gradient magnitude  $||\nabla f(x_{n_k})||_2$ can be made arbitrarily small. In other words, for any $\epsilon > 0$, there exists $m \in \mathbb{N}$ such that $||\nabla f(x_m)||_2 \leq \epsilon$. 
\end{proof}

Note that in this artificial case, we also need the norm of the $||x_k||$ to diverge to infinity because otherwise, the existence of a convergent subsequence implies convergence of the entire sequence as before. 

While we can show that in the case where there are infinitely many sparsity pattern changes, the IWHT algorithm generates points with arbitrarily small gradient norms from time to time while also maintaining a monotonically decreasing objective, it is not immediate that we have convergence to any point without an additional assumption like boundedness of $(x_k)$.

For the CIWHT Algorithm, we cannot use the technique used in the proof of Theorem \ref{thm:L_convergence} to prove that if a limit point exists then the entire sequence converges to that point. Instead, we can argue using Theorem \ref{thm:sparsity_pattern_change}. If the number of sparsity pattern changes is finite then eventually, the sparsity pattern is identified and $f(x_k)$ converges to $f(\Bar{x})$ where $\Bar{x}$ is some optimal point corresponding to the last sparsity pattern. If it is infinite, it runs into the same issue as before where an approximately global optimal solution is produced. However, if we assume that sequence $(x_k)$ is bounded then these points will have a convergent subsequence and that point will have a zero gradient norm which will be a global optimum. Furthermore, $f(x_k)$ sequence will converge to the same objective, thus $(x_k)$ sequence will either converge to that point or jump between global optimum points. 

An immediate consequence of the previous theorems is the convergence rate of the IWHT algorithm.

\begin{corollary}
\label{cor:convergence_rate}
Let $(x_k)$ be a sequence generated by the IWHT algorithm. Assume that sequence $(x_k)$ is bounded. Then for sufficiently large $k$, the function value converges with $O(1/k)$. If $f(x)$ is restricted strongly convex, then it converges linearly. 
\end{corollary}

\begin{proof}
The proof follows\footnote{In the artificial case that $x_k$ converges to a point with $||x||_0 < s$, proofs do not follow immediately from the classical arguments. Additional inequalities are needed. The key inequality for the standard convex case to prove $O(1/k)$  convergence rate is derived in equation (\ref{eq:bregman_descent_inequality6}) and for the linear convergence, we give a brief derivation in Appendix \ref{ap:inequality_linear_conv_strong_conv}.} from the fact that from the previous theorem sequence $(x_k)$ converges and thus, for sufficiently large $k$, $(x_k)$ will not change the sparsity pattern. Then the IWHT algorithm reduces to standard gradient descent in the restricted space and hence the usual gradient descent rate of $O(1/k)$ in the Lipschitz gradient case \cite{teboulle2023elementary} and $O(\rho^k)$ in the additional strongly convex case follow \cite{nesterov2018lectures}.
\end{proof}

\subsection{CW-minimum}
In \cite{amir_beck_sparse_opt} the concept of CW-minimality (CW-optimality)  is introduced for optimization over the set of sparse vectors. 
\begin{definition}
Let $x^* \in C_s$ be a sparse vector. $x^*$ is called a CW-Minimum point of $f$ if it satisfies 

Either $||x^*||_0 < s$ and the following holds
\begin{equation}
\label{eq:cw_minimum_1}
\begin{array}{llll}
f(x^*) = \underset{t\in \R}\min f(x^* + t e_i) \quad \forall i = 1,..,n, \\
\end{array}
\end{equation}
or $||x^*||_0 = s$ and one has \\
\begin{equation}
\label{eq:cw_minimum_2}
\begin{array}{llll}
f(x^*) \leq \underset{t\in \R}\min f(x^* - x^*_ie_i + t e_j) \quad \forall i \in I_1(x^*), j = 1,..,n.
\end{array}
\end{equation}
\end{definition}
Intuitively, CW checks for pairwise change of sparsity pattern without optimizing over the new space. Instead, only optimization over the new variable is performed. Clearly, a globally optimal solution satisfies this property. Authors also show that if a point satisfies CW-Minimum property then it also satisfies $L$-stationarity \ref{theorem:l_stationarity} \cite{amir_beck_sparse_opt}. We generalize this result to Bregman functions that have a separable form. 

\begin{theorem}
Let $x \in C_s$ be a CW-minimum point. Then it is $L_h$-Bregman stationary with respect to any separable Bregman function.
\end{theorem}
\begin{proof}
We will prove this statement by its contrapositive. If a point is not Bregman stationary then it is not a CW-Minimum either. We analyze the situation in three cases:

Case 1. $||x||_0 \leq s$ but $\nabla_i f(x) \neq 0$ for some $i \in I_1(x)$. Then, using case 1 (\ref{eq:cw_minimum_1}) or case 2 from (\ref{eq:cw_minimum_2}), minimizing $f(x)$ over $x_i$ gives a smaller value and hence it is not a CW-Minimum point.

Case 2. $||x||_0 < s$ and $\nabla_i f(x) = 0 \ \forall i \in I_1(x)$. If the point $x$ is not Bregman stationary, then $\nabla f \neq 0$. Then, the condition in (\ref{eq:cw_minimum_1}) would set some $j \in I_0(x)$ to nonzero value where $\nabla_j f \neq 0$. Thus, $x$ is not a CW-minimum point. 

Case 3.  $||x||_0 = s$ and $\nabla_i f(x) = 0 \ \forall i \in I_1(x)$. If the point $x$ is not Bregman stationary, then the algorithm described in (\ref{eq:bregman_descent_seperable_solution}) does not pick all indices of $I_1(x)$. Hence, there exist $j \in I_0(x)$ and $u_j$ such that $\phi_j(0)-\phi_j(u_j) > \phi_i(0)-\phi_i(u_i)$ for some $i \in I_1(x)$. Then, the condition in (\ref{eq:cw_minimum_2}) implies that point is not a CW-Minimum because the minimization over $x_j$ while setting $x_i = 0$  bounded above by $t = u_j$.
\end{proof}

We also note that we can interpret the CW Minimality as a Bregman stationarity with the choice of kernel generating distance $h(x) = f(x) + M\sum_{k \neq i,j} x_k^2$ for sufficiently large $M > 0$. Then we can take $L_h = 1$ and the upper bound on $f(x)$ becomes
\begin{equation}
\label{eq:cw_min_bregman}
\begin{array}{ll@{}ll}
f(y) &\leq f(x) + \nabla f(x)^T(y-x) + D_h(y,x) \\
 &= f(x) + \nabla f(x)^T(y-x) + f(y) - f(x) -  \nabla f(x)^T(y-x) + M\underset{k \neq i,j}\sum (y_k-x_k)^2 \\ 
 &= f(y) + M\underset{k \neq i,j}\sum (y_k-x_k)^2.
\end{array}
\end{equation}
For sufficiently large $M$, to minimize this upper bound given in (\ref{eq:cw_min_bregman}), for any $k \neq i,j$ we have to set $y_k = x_k$. Then assuming $i \in I_1(x), j \in I_0(x)$, we can either set $y_i = x_i$ and keep $y_j = x_j = 0$, hence changing nothing, or we can set $y_i = 0$ and optimize over $y_j$. This is exactly described in (\ref{eq:cw_minimum_2}), but for all possible choices of pairs of $(i,j)$. This gives a partial converse intuition for the connection between CW-Minimality and Bregman Stationarity.

\section{Computing $H$ or $D$}
\label{sec:compute_weights}
In this section, we briefly discuss how to compute a Diagonal Scaling Matrix (DSM) $D$ such that $f(x)$ is relatively smooth to $h(x) = \frac{1}{2}x^TDx$.

One can see easily that if a DSM $D$ majorizes the Hessian\footnote{Although we do not need twice differentiability to use IWHT, it simplifies the procedure to compute a DSM.} of $f$, everywhere, i.e. $D \succeq \nabla^2 f(x)$, then it is also true that $D+A$, where $A \succeq 0$ also majorizes Hessian of $f$ everywhere. This is analogous to the Lipschitz gradient, where the problem of interest is often to find the smallest Lipschitz constant for the gradient of $f$ since using larger Lipschitz constants makes the algorithm more conservative and slower for convergence-like results. Furthermore, if we choose the Lipschitz constant sufficiently large, any local minimum of (\ref{eq:optimization_model}), will satisfy $L$-stationarity condition \ref{theorem:l_stationarity}, which makes the algorithm very pessimistic in terms of possible points it can converge to. Therefore, we want a DSM that is "small" in some sense. However, since a DSM uses more than one parameter the best criterion is not immediately obvious for choosing a minimal DSM.

To that end, we propose three models to compute a DSM. Firstly, the model with linear cost, (\ref{eq:compute_D_L}) denoted with $D_L$, the model with quadratic cost (\ref{eq:compute_D_Q}), denoted with $D_Q$ and the minimax model (\ref{eq:compute_lipschitz}) which we denote with $L$ since it coincides with the Lipschitz constant.

\begin{minipage}[t]{0.45\textwidth}
\begin{equation}
\label{eq:compute_D_L}
\begin{array}{ll@{}ll}
\underset{w}\min &\quad {\bf 1}^T w  \\
\text{st.} & \Diag(w) \succeq C \\
\end{array}
\end{equation}
\end{minipage}
\hfill 
\begin{minipage}[t]{0.45\textwidth}
\begin{equation}
\label{eq:compute_D_L_dual}
\begin{array}{lllll}
\underset{Z}\max & \langle C, Z\rangle  \\
\text{st.} & Z_{ii} = 1 \quad \forall i \\
& Z  \succeq 0 \\
\end{array}
\end{equation}
\end{minipage}

\begin{minipage}[t]{0.45\textwidth}
\begin{equation}
\label{eq:compute_D_Q}
\begin{array}{lllll}
\underset{w}\min & \quad \frac{1}{2}w^T w  \\
\text{st.} & \Diag(w) \succeq C \\
\end{array}
\end{equation}
\end{minipage}
\hfill 
\begin{minipage}[t]{0.45\textwidth}
\begin{equation}
\label{eq:compute_D_Q_dual}
\begin{array}{lllll}
\underset{Z}\max & \langle C,Z\rangle -\frac{1}{2} k^Tk  \\
\text{st.} & Z_{ii} = k_i \quad \forall i \\
& Z  \succeq 0 \\
\end{array}
\end{equation}
\end{minipage}

\begin{minipage}[t]{0.45\textwidth}
\begin{equation}
\label{eq:compute_lipschitz}
\begin{array}{ll@{}ll}
\underset{w}\min & \underset{i}\max \quad w_i  \\
\text{st.}  & \Diag(w) \succeq C \\
\end{array}
\end{equation}
\end{minipage}
\hfill 
\begin{minipage}[t]{0.45\textwidth}
\begin{equation}
\label{eq:compute_lipschitz_dual}
\begin{array}{lllll}
\underset{Z}\max & \langle C,Z\rangle  \\
\text{st.} & \tr(Z) = 1 \\
& Z  \succeq 0. \\
\end{array}
\end{equation}
\end{minipage}

In these models, for computational purposes, we assume that we are given some matrix $C$ such that $C \succeq \nabla^2 f(x) \ \forall x$. For instance, in the least squares type problems $C = A^TA$ can be taken since the Hessian is constant, and in the logistic regression type problems, it can be $C = \frac{1}{4}A^TA$ which is the worst case largest Hessian when the diagonal matrix of probabilities $P$ is taken to be the worse case, i.e $P = \frac{1}{2}I$ so that $P(I-P)$ is maximized. Additionally, for the log-sum-exp type of problems, the Hessian  is in the form of $Diag(w)-ww^T$ for some $w\geq 0, 1^Tw = 1$. Then, the Hessian is bounded above by $\frac{1}{2}I_n$, where $I_n$ is the identity matrix. 

Additionally, these models can be relaxed. If Algorithm \ref{alg:ciwht} is used to compute a sparse solution for (\ref{eq:optimization_model}), then $\Diag(w) \succeq C$ is a conservative constraint. This is because optimization can instead be done over $w$ such that $2s$ sparse minimum eigenvalue of $\Diag(w)-A$ is non-negative since $D$ is multiplied by a vector that has cardinality at most $2s$. This new set is more relaxed than $\Diag(w) \succeq A$. Nevertheless, this constraint is not easy to work with since checking the feasibility of a given $w$ is even NP-Hard as the problem boils down to the Sparse Principal Component Analysis (SPCA) problem \cite{spca_np_hard}. 

Computationally, we use the dual of these models to compute the weights. In Appendix \ref{ap:get_primal} we briefly show how to compute primal variables from a dual solution. Although these problems can be solved with convex optimization methods, interior point methods do not scale well to larger-scale problems. A common approach to alleviate this problem is to use Burer-Monteiro (BM) factorization \cite{burer_monteiro}. BM factorization stores the positive semidefinite matrix $Z$ implicitly by using a low-rank factorization $Z = BB^T$, where $B \in \mathbb{R}^{n\times k}$ and $k$ is some rank parameter chosen. When the algorithm operates on the $B$ matrix instead of $Z$, the semidefinite cone constraint can be removed from the problem. Furthermore, if $k$ is chosen much smaller than $n$, the algorithm will use less storage.  In the original paper, the authors use this factorization with an Augmented Lagrangian Method (ALM). Subsequent research has demonstrated that substituting the ALM with Block-Coordinate Maximization (BCM) \cite{javanmard2016phase,mixing_method,cnv_BM_SDP}, Riemannian gradient \cite{javanmard2016phase,mei2017solvingsdpssynchronizationmaxcut}, or Riemannian trust-region \cite{absil_rtr,boumal2016non,journee2010low} methods yields superior empirical performance. In particular, the BCM algorithm proposed in \cite{cnv_BM_SDP,mixing_method} is shown to be faster than state-of-the-art methods for problems in the form of (\ref{eq:compute_D_L_dual}). 

These results lead us to adapt the BCM algorithm proposed in \cite{cnv_BM_SDP} to solve models in (\ref{eq:compute_D_L_dual}) and (\ref{eq:compute_D_Q_dual}). 

\begin{algorithm}[ht]
\caption{BCM for Model (\ref{eq:compute_D_L_dual})}
\label{alg:bcm_L}
\begin{algorithmic}[1]
\STATE{\textbf{Input:} $C, B$}
\FOR{$i = 1,2, ...,$}
    \FOR{$j = 1,2, ...,n$}
        \STATE{$u_j = C_j - C_{jj}e_j$}
        \STATE{$g_j = B^Tu_j $}
        \STATE{$B_j = \frac{g_j}{||g_j||_2}$}
    \ENDFOR
\ENDFOR
\end{algorithmic}
\end{algorithm}

\begin{algorithm}[ht]
\caption{BCM for Model (\ref{eq:compute_D_Q_dual})}
\label{alg:bcm_Q}
\begin{algorithmic}[1]
\STATE{\textbf{Input:} $C,B $}
\FOR{$i = 1,2, ...,$}
    \FOR{$j = 1,2, ...,n$}
        \STATE{$u_j = C_j - C_{jj}e_j$}
        \STATE{$g_j = B^Tu_j $}
        \STATE{Find the (unique) positive root of $t^3-C_{jj}t-||g||_2 = 0$ }
        \STATE{$B_j = t\frac{g_j}{||g_j||_2}$}
    \ENDFOR
\ENDFOR
\end{algorithmic}
\end{algorithm}

We note that Model (\ref{eq:compute_lipschitz_dual}) can be solved easily by a maximum eigenvalue solver without using a complicated optimization algorithm. 
 
Alternatively, instead of relaxing the positive semidefinite cone constraint with BM factorization, these problems can also be solved in $Z$ space by using convex optimization techniques that scale well to large-scale problems like Frank-Wolfe type algorithms \cite{yurtsevercgal,yurtsever19}. Nevertheless, computational results shown in respective papers as well as in \cite{cnv_BM_SDP,mixing_method} motivated us to adapt the coordinate descent type algorithm. 

Finally, in our numerical tests, we observed that relaxing the Algorithms \ref{alg:bcm_L} and \ref{alg:bcm_Q} by parallelizing the inner loop for computational speed-up works well. This makes the Algorithm \ref{alg:bcm_L} reduce to the Conditional Gradient (CG) algorithm with a unit step size. On the other hand Algorithm \ref{alg:bcm_Q} becomes a mixture of the CG algorithm (calculating $g_j$ for each row) and Coordinatewise Maximization (computing value of $t$ for each row).

\section{Numerical Results}
\label{sec:numerical_results}
In this section, we present the algorithmic choices made during the numerical implementations of the algorithms.

\subsection{Algorithmic Considerations}
We begin by discussing the line search procedure described in Algorithm \ref{alg:line_search} below.
\begin{algorithm}[ht]
\caption{Line Search Strategy}
\label{alg:line_search}
\begin{algorithmic}[1]
\STATE{\textbf{Input:} $D, x \in C_s, \alpha \in (0,1) ,J \in \mathbb{N}, \beta > 0$}
\FOR{$j = 1,2,.., J$}
    \STATE{$D_j = \alpha^{J-j}D$}
    \STATE{$y_{j} = \mathcal{P}_{C_s}(D_j^{\frac{1}{2}}x-D_j^{-\frac{1}{2}}\nabla f(x))$}
    \STATE{$x_{j} = D_j^{-\frac{1}{2}}y_{j}$}
    \IF{$f(x_j) \leq f(x) - \beta ||x_j-x||_2^2$}
    \STATE{\textbf{break}}
    \ENDIF
\ENDFOR
\RETURN{$x_j$}
\end{algorithmic}
\end{algorithm}


For the framework of (\ref{eq:optimization_model}), without line search, if we only take sufficiently small step sizes, then the algorithm is still a descent algorithm but often generates mediocre quality points. Line search usually improves the gradient descent type algorithms \cite{nocedal1999numerical,bertsekas1997nonlinear,first_order_amir_beck} and in the sparsity constrained problems helps discover superior quality points in terms of objective function value. Line search was further motivated in Section \ref{sec:compute_weights}, with how we compute the diagonal matrix $D$ that majorizes $\nabla^2 f(x)$ in practice. Furthermore, Theorem \ref{thm:L_convergence} shows that often in a well-conditioned problem, we will have convergence very quickly. Thus, we will not search through many points, which is desirable since Model (\ref{eq:optimization_model}) is a difficult non-convex problem. This consideration motivates restarting as described in Algorithm \ref{alg:sparse_restart}. These two ingredients combined improve the basic algorithms described in \ref{alg:ciwht} and \ref{alg:iwht} remarkably as can be seen in Figure \ref{fig:ciwht_effect_recovery}.

\begin{algorithm}[ht]
\caption{Gradient Based Restart for Extended Search}
\label{alg:sparse_restart}
\begin{algorithmic}[1]
\STATE{\textbf{Input:} $D, x \in C_s, \gamma \in (0,1)$}
\STATE{$\widehat{D} = \gamma D$}
\STATE{$y = \mathcal{P}_{C_s}(\widehat{D}^{\frac{1}{2}}x-\widehat{D}^{-\frac{1}{2}}\nabla f(x))$}
\STATE{$x = \widehat{D}^{-\frac{1}{2}}y$}
\end{algorithmic}
\end{algorithm}

In \cite{gpnp}, a complex search algorithm that combines a line search strategy Algorithm \ref{alg:line_search}, with a restarting scheme for the search based on gradient information as in Algorithm \ref{alg:sparse_restart}, and Newton steps described in Algorithm \ref{alg:newton_step} to accelerate the convergence  is proposed\footnote{The restarting scheme is not mentioned in \cite{gpnp}. However, it is present in the code distributed by the author.}. Thus, it performs a significantly more extensive search.   Essentially, the algorithm attempts a classical Newton step in the restricted space where $x_i$'s are non-zero. The step is accepted if the objective decreases by some amount, otherwise, it is rejected. 

\begin{algorithm}[ht]
\caption{Newton Step}
\label{alg:newton_step}
\begin{algorithmic}[1]
\STATE{\textbf{Input:} $x \in C_s, \beta > 0$}
\STATE{Solve $\nabla_{I_1(x),I_1(x)}^2f(x)(v_{I_1(x)}-x_{I_1(x)}) = -\nabla_{I_1(x)} f(x), \quad  v_{I_0(x)} = 0$}
\IF{$f(v) \leq f(x) - \beta ||v-x||_2^2$}
\STATE{$x = v$}
\ENDIF
\end{algorithmic}
\end{algorithm}

Without going into much detail we numerically motivate the first-order acceleration methods for the problem (\ref{eq:optimization_model}) used in our experiments. We achieve this acceleration by adapting a monotone version of the fast iterative shrinkage thresholding algorithm (MFISTA) proposed in \cite{mfista}. In our preliminary numerical testing, we found the behavior of MFISTA and its non-monotone version to be very similar, yet MFISTA is more interesting theoretically for (\ref{eq:optimization_model}) because it keeps the monotonicity property in terms of objective as in proved for IWHT algorithm in Lemma \ref{lemma:subsequential_convergence}.

\begin{algorithm}[ht]
\caption{MFISTA for (\ref{eq:optimization_model})}   
\label{alg:mfista}
\begin{algorithmic}[1]
\STATE{\textbf{Input:} $x_0 \in C_s, \beta > 0$}
\STATE{\textbf{Initilization:} $y_0 = x_0,t_0 = 1$}
\FOR{$k = 1,2, ...,$}
\STATE{pick $\alpha_k$ and $D_k$ and set $\widehat{D} = \alpha_k D_k$}
\STATE{$z_{k} = \widehat{D}^{-\frac{1}{2}}\mathcal{P}_{C_s}(\widehat{D}^{\frac{1}{2}}y_{k}-\widehat{D}^{-\frac{1}{2}}\nabla f(y_{k}))$}
\IF{$f(z_k) \leq f(x_k)$}
\STATE{$x_{k+1} = z_k$}
\ELSE
\STATE{$x_{k+1} = x_k$}
\ENDIF
\IF{$I_1(z_k) = I_1(x_k)$}
\STATE{$t_{k+1} = \frac{1 + \sqrt{1 + 4t_k^2}}{2}$}
\STATE{$y_{k+1} = x_{k+1} + \frac{t_k}{t_{k+1}}(z_k-x_{k+1}) + \frac{t_k-1}{t_{k+1}}(x_{k+1}-x_k) $}
\ELSE
\STATE{$t_{k+1} = 1, y_{k+1} = x_{k+1}$}
\ENDIF
\ENDFOR
\end{algorithmic}
\end{algorithm}

\subsection{Numerical Experiments}
This section focuses on the least squares problem with a sparsity constraint. Thus, our optimization model of interest looks as follows
\begin{equation}
\label{eq:sparse_least_squares}
\begin{array}{ll@{}ll}
\underset{x}\min \quad & \frac{1}{2} ||Ax-b||_2^2  \\
\text{st.} \quad & ||x||_0 \leq s. \\
\end{array}
\end{equation}
This model attracted a lot of research interest after the development of the theory of compressed sensing (CS) \cite{candes2006robust,donoho2006compressed,candes2005decoding}. This was followed up by many algorithmic developments \cite{blumsignal,blumdaviesharmonic,blumdavies-selected,gpnp,needell2009cosamp,bahmani2013greedy,pati1993orthogonal,candes2005l1,yang2011alternating,figueiredo2007gradient}. For an extensive numerical comparison between various types of algorithms, we refer to \cite{gpnp,nhtp}.  

We follow the benchmark used in \cite{gpnp}. Entries of the data matrix $A \in \mathbb{R}^{64\times 256}$ are generated randomly from samples of independent standard Gaussian distribution $\mathcal{N}(0,1)$. Then the columns of the generated random matrix are scaled to have a unit norm. The $s$-sparse ground truth signal $x^*$ is also generated with $\mathcal{N}(0,1)$. Finally, we compute the observation vector $b = Ax^*$. Thus, we are testing our algorithms in the noiseless recovery problems, and therefore the magnitude of $x^*$ does not affect the problem. If an algorithm computes a solution $x$ with a small relative error, i.e. $\frac{||x-x^*||}{||x^*||} < 10^{-4}$ then we say recovery is successful. We will refer to this simulation shortly as the CS problem. 

To abbreviate and differentiate the algorithms, we use Newton + $L$ to refer to the GPNP algorithm \cite{gpnp}. Newton + $D_Q + D_L + L$  refers to the modified GPNP algorithm where we change the simple line search strategy by incorporating different DSM's computed by the models given in (\ref{eq:compute_D_L}),(\ref{eq:compute_D_Q}) and (\ref{eq:compute_lipschitz}). Each DSM is applied for one iteration and changed/swapped for another in the next iteration. It also shows the order of DSM's applied in the algorithm. IHT refers to \ref{alg:iwht} where $D = L_f I$, and CIWHT refers to \ref{alg:ciwht} with the same three DSM's computed as before.   LS\& R implies the incorporation of Line Search \ref{alg:line_search} and gradient-based restart \ref{alg:sparse_restart}, instead of using theoretically safe choices. Convention in the naming also applies to the version of the MFISTA. Table \ref{tab:alg_composition} summarizes the composition and naming scheme of the algorithms\footnote{This is not an exhaustive list but rather an overview of the naming of the algorithms. Implementations of the algorithms in Python are available in \url{https://github.com/Fatih-S-AKTAS/iwht}}.

\begin{table}[ht]
\centering
\scalebox{0.8}{ 
\begin{tabular}{|c|c|c|c|c|c|c|c|}
\hline
Framework \textbackslash Architecture & Algorithm \ref{alg:line_search} & Algorithm \ref{alg:sparse_restart} & Algorithm \ref{alg:newton_step} & $D_Q$  & $D_L$ & $L$ \\ \hline
IHT & \xmark & \xmark & \xmark & \xmark & \xmark & \cmark \\ \hline
IWHT & \xmark & \xmark & \xmark & \cmark & \cmark & \cmark \\ \hline
IHT + LS\&R & \cmark & \cmark & \xmark & \xmark & \xmark & \cmark \\ \hline
CIWHT + LS\&R & \cmark & \cmark & \xmark & \cmark & \cmark & \cmark \\ \hline
\makecell{Newton + $L$ (GPNP)} & \cmark & \cmark & \cmark & \xmark  & \xmark & \cmark \\ \hline
\makecell{Newton + $D_Q$ \\ + $D_L$ + $L$} & \cmark & \cmark & \cmark & \cmark  & \cmark & \cmark \\ \hline
\end{tabular}}
\caption{Compositions of the algorithms summarized}
\label{tab:alg_composition}
\end{table}

We set the iteration number to be 15000 for all algorithms, measured by a gradient descent-like step or the line search step. Additional stopping criteria used in \cite{gpnp} is also adapted for the GPNP-like algorithm that uses the Newton steps. However, we noticed that these stopping criteria often hurt the recovery rates of algorithms that do not use the Newton steps. Thus we only stop an algorithm that does not use Newton steps if a computed solution satisfies $||Ax-b|| \leq 10^{-10}$. For the GPNP type algorithms, we also use the criterion that if the objective value of the last $5$ points generated has a standard deviation of less than $10^{-10}$, then we terminate the search. 

First, we generated 200 instances of the CS problem and compared the effect of line search and restart, using the three DSMs and the Newton steps on the recovery rate. Figure \ref{fig:ciwht_effect_recovery} demonstrates the results.  

\begin{figure}[ht]
    \centering
    \includegraphics[width=\textwidth]{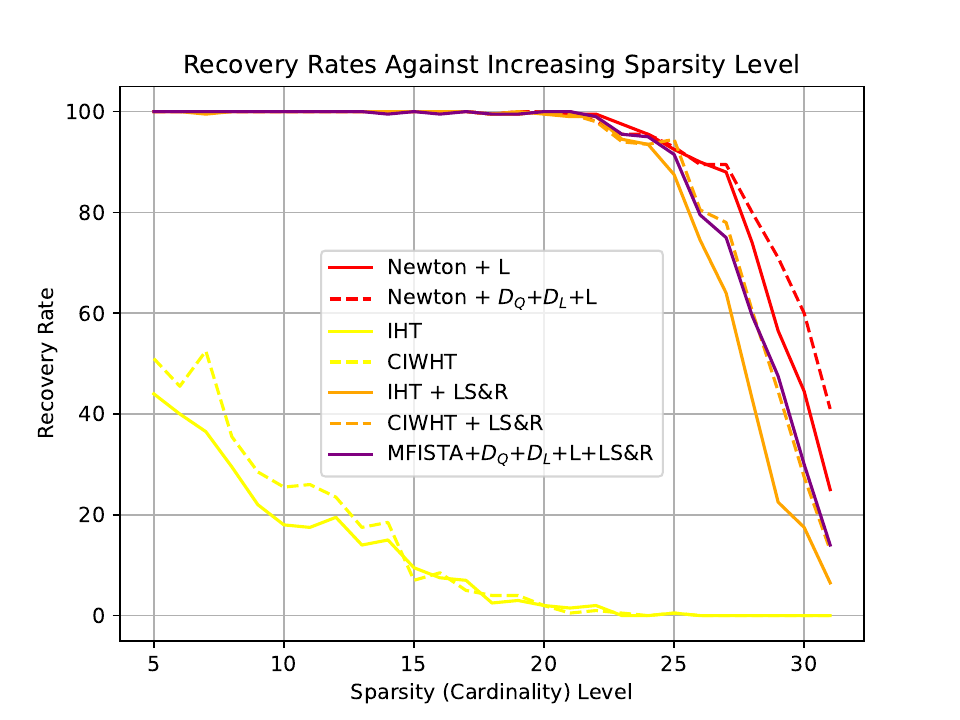}
    \caption{Comparison of different algorithmic frameworks}
    \label{fig:ciwht_effect_recovery}
\end{figure}

Figure \ref{fig:ciwht_effect_recovery} shows that the line search strategy and gradient-based restart remarkably improve the recovery rates. Moreover, using DSM instead of simple step sizes also enhances the recovery rates independent of the framework used. In other words, any algorithm equipped with DSM has a higher recovery rate than the version of it that uses only simple step size. Further incorporating the Newton steps seems to increase the recovery rates. Our intuition behind these results is that DSMs allow the algorithm to escape suboptimal points because a locally optimal point may be $L$-stationary but not $D_L$ or $D_Q$ stationary. We believe the increase in recovery rates obtained by using the Newton step is by acceleration, which allows the algorithm to converge faster, which in turn allows the overall algorithm to do a more extensive search since we also use gradient-based restart for the algorithms that use the Newton steps.

Next, to compare the effect of using different combinations of DSMs and different orders of application in the gradient steps, we conducted more extensive computational experiments utilizing the GPNP framework. Using the same setup we generated 500 instances of the CS problem. Also, in this experiment, we denote by $P$ the number of iterations after which we change the DSM used, short for the period to test out the effect of changing the DSM frequently. 

\begin{figure}[ht]
    \centering
    \includegraphics[width=\textwidth]{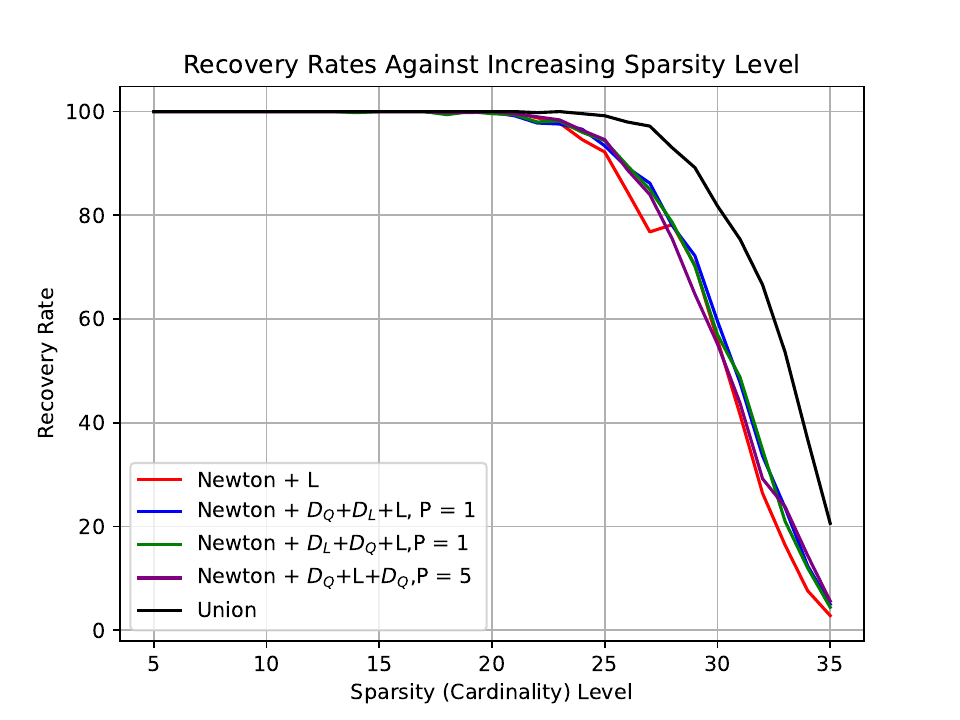}
    \caption{Effect of choice of DSM in GPNP framework}
    \label{fig:newton_comparison}
\end{figure}

``Union" is a conceptual algorithm where we use all algorithms and consider recovery successful if any of them recovers the ground truth. Figure \ref{fig:newton_comparison} shows that using the Lipschitz constant (or uniform step size since line search is done) is in fact a reasonable choice. When the sparsity level is close to 24, we see that the simple GPNP algorithm does slightly better than other algorithms while falling behind after cardinality level 27.  This shows that, contrary to our intuition, the algorithms that use multiple DSMs do not uniformly dominate the GPNP which uses only the Lipschtiz constant. This point can also be seen in the more difficult problems with high cardinality level $s$ with a fixed number of observations. The conceptual union algorithm has a significantly higher recovery rate than any of the algorithms, which shows that each algorithm performs well on different types of problems. Otherwise, the recovery rate of the union algorithm should have been much closer to the individual recovery rates of the algorithms. More detailed results of this experiment are reported in Appendix \ref{ap:additional_numerical}.

\begin{figure}[ht]
\centering
\begin{subfigure}{.5\textwidth}
  \centering
  \includegraphics[width=1\linewidth]{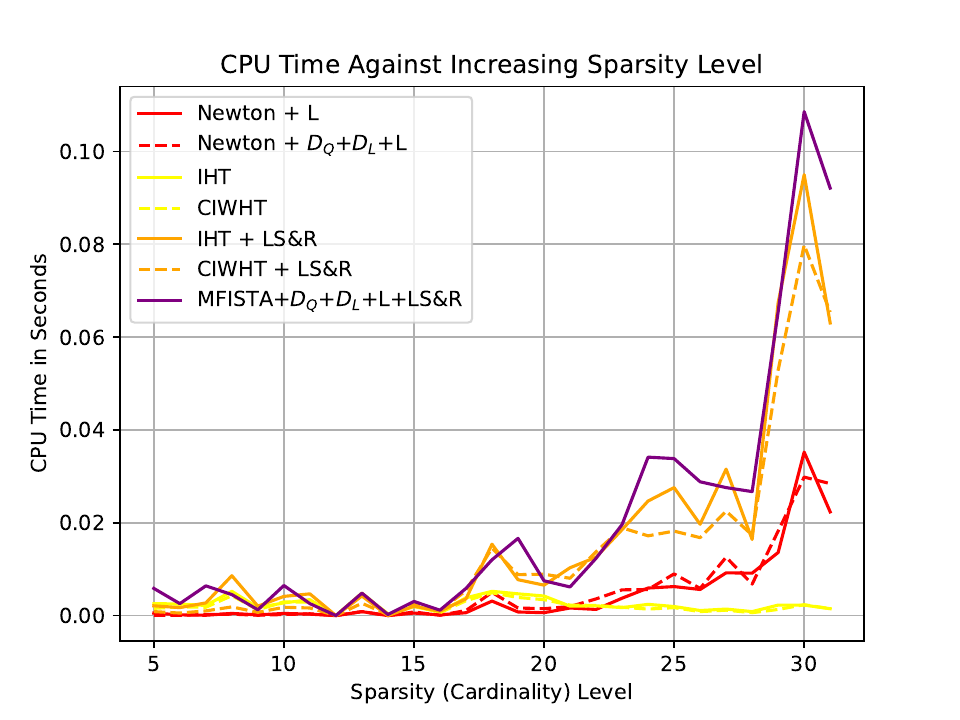}
  \caption{CPU Time}
  \label{fig:ciwht_effect_cpu}
\end{subfigure}%
\begin{subfigure}{.5\textwidth}
  \centering
  \includegraphics[width=1\linewidth]{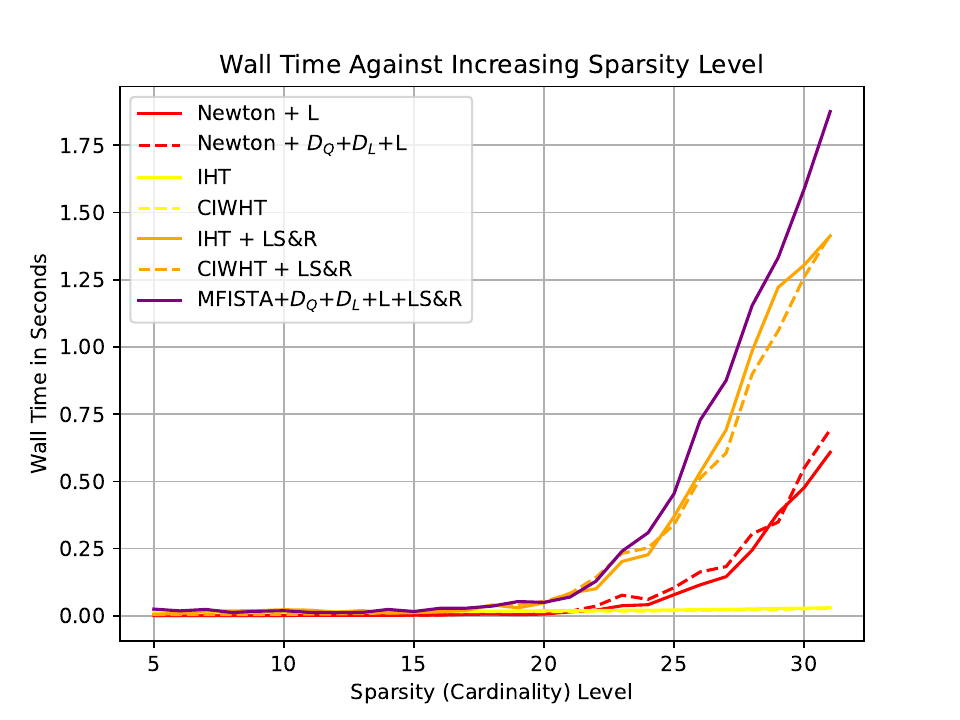}
  \caption{Wall Time}
  \label{fig:ciwht_effect_wall}
\end{subfigure}
\caption{Computational cost of different algorithmic frameworks}
\label{fig:ciwht_effect_computation}
\end{figure}

As shown in the Figures \ref{fig:ciwht_effect_computation}, \ref{fig:newton_comparison_computation}  the additional computational cost of applying multiple DSMs seems to be negligible. In fact, sometimes even beneficial because of the stopping criteria we used, sometimes recovery happens fast and algorithms terminate early as observed in Figure \ref{fig:ciwht_effect_computation}. However, Figure \ref{fig:newton_comparison_computation} shows the reverse effect also happens sometimes, for very difficult problems when the sparsity level is high, applying different DSMs makes the algorithm search for longer while applying only the Lipschitz constant hits a stopping criterion early. For other sparsity levels, there seems to be no significant difference in computational time between using multiple DSMs or only a single one. However, we remark that the additional cost of computing DSMs is not included in the computational costs of the algorithms. We assume it is computed separately, before applying the algorithms. The computational cost of computing DSMs is reported in subsection \ref{sec:numerical_compute_weights}.

Furthermore, although the computational burden of introducing a line search scheme instead of using a theoretically safe step size is relatively high, the algorithms still run reasonably fast. In addition, we see the benefit of incorporating Newton steps to accelerate the algorithm in the running times. Nevertheless, the MFISTA algorithm equipped with multiple DSMs does not improve the algorithm's recovery rate or achieve acceleration like Newton steps. We believe this could be caused by changing the DSM used in the proximal gradient step too frequently, which impedes the acceleration since the norm changes every iteration. 

\begin{figure}[ht]
\centering
\begin{subfigure}{.5\textwidth}
  \centering
  \includegraphics[width=1\linewidth]{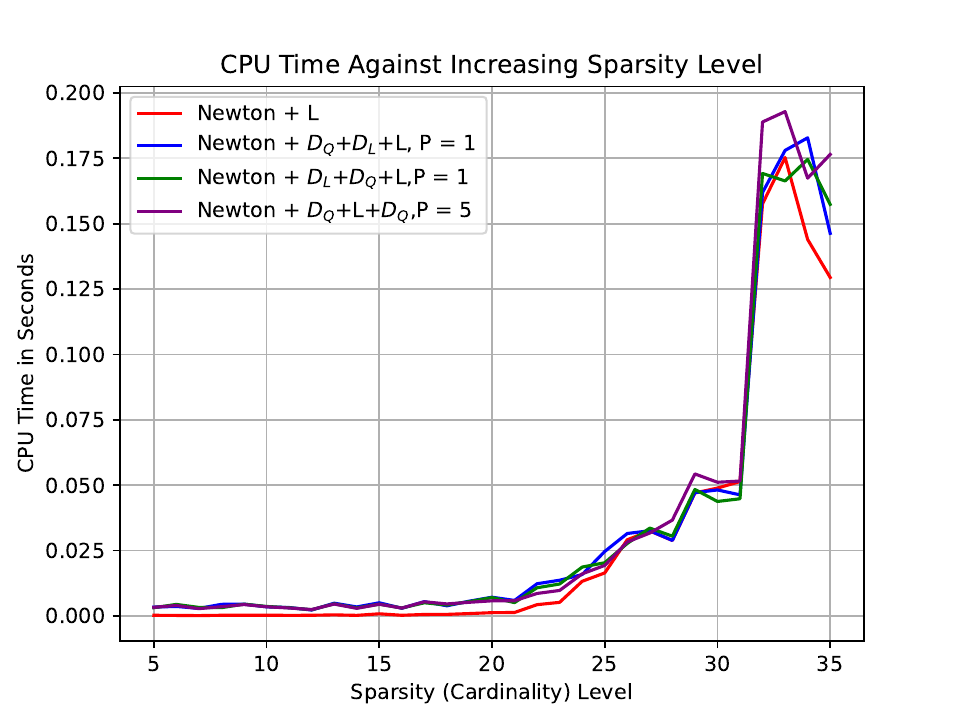}
  \caption{CPU Time}
  \label{fig:newton_comparison_cpu}
\end{subfigure}%
\begin{subfigure}{.5\textwidth}
  \centering
  \includegraphics[width=1\linewidth]{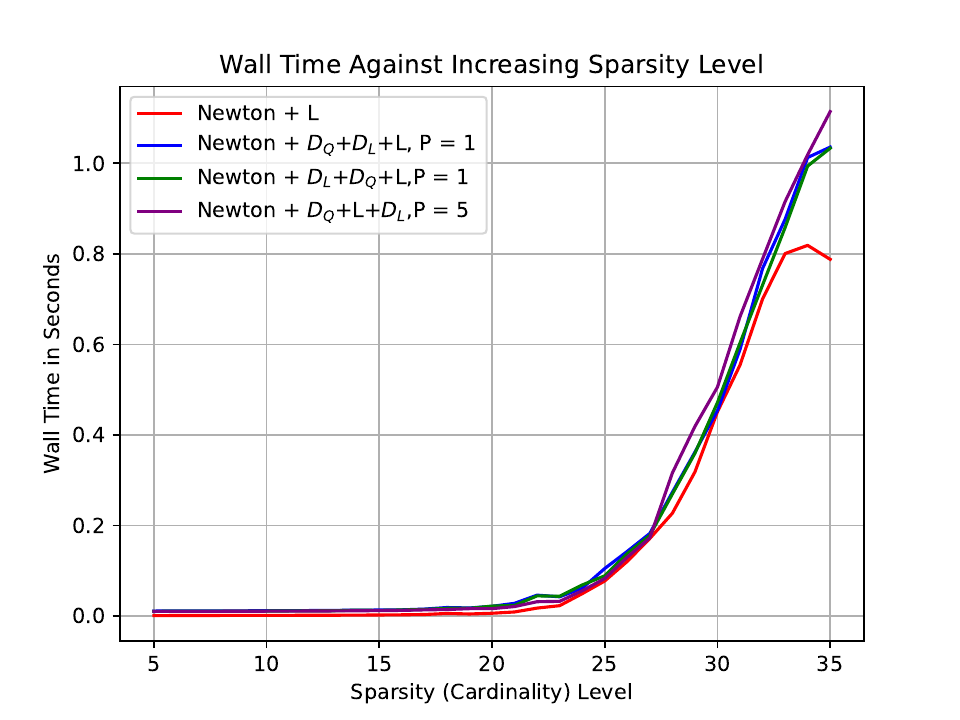}
  \caption{Wall Time}
  \label{fig:newton_comparison_wall}
\end{subfigure}
\caption{Computational costs of GPNP algorithms equipped with different DSMs}
\label{fig:newton_comparison_computation}
\end{figure}

Finally, we visualize the sequence of $f(x_k)$ generated by different algorithms for three cardinality levels with the same setup. Figure \ref{fig:f_x_trajectory} shows rapid linear convergence for all algorithms, illustrating the convergence rate proved in Corollary \ref{cor:convergence_rate}, likely because randomly generated Gaussian matrix satisfies the restricted strong convexity property. Naturally, the algorithms that apply a line search instead of theoretical safe step size choices converge much faster empirically. Moreover, we can also recognize the sparsity pattern changes in the sudden change in the $f(x_k)$ as well. Instead of drawing a smooth curve that decreases down to a limit, if there is a sudden substantial decrease, it is a sparsity pattern change and the objective value decreases sharply at that iteration. Also when the search algorithm is restarted, it can be recognized by the large spikes in $f(x_k)$. For the sparsity level $s = 16$, all algorithms except CIWHT converge to the optimal solution without needing a restart. For larger cardinality levels, algorithms start to behave slightly differently, especially after the first restart, although the early behavior seems to be similar. 

\begin{figure}[ht]
\centering
\includegraphics[width=.32\textwidth]{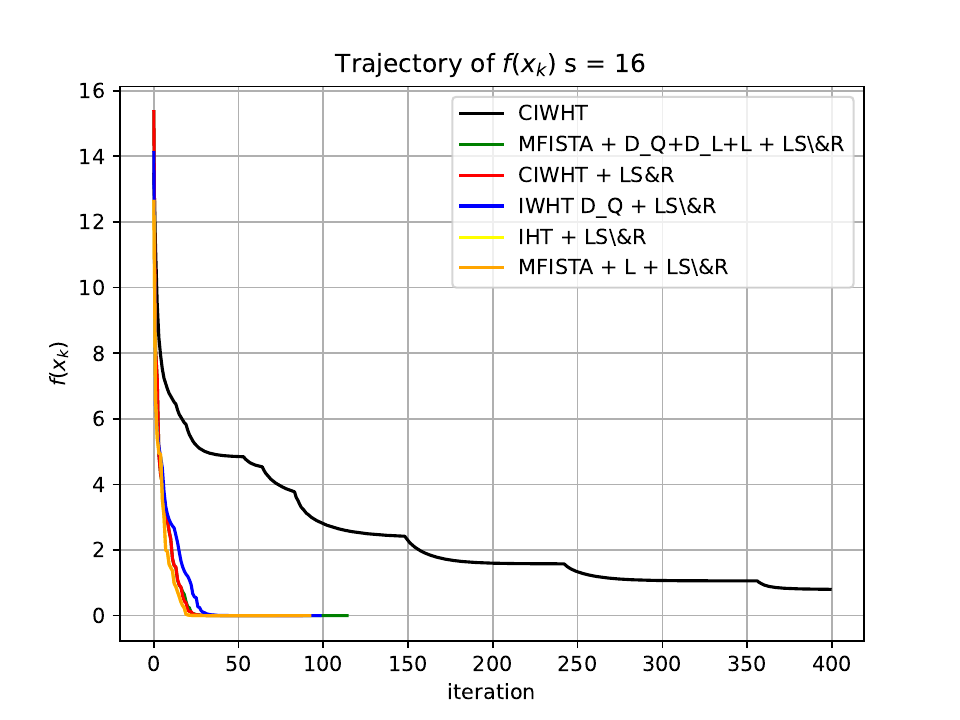}
\includegraphics[width=.32\textwidth]{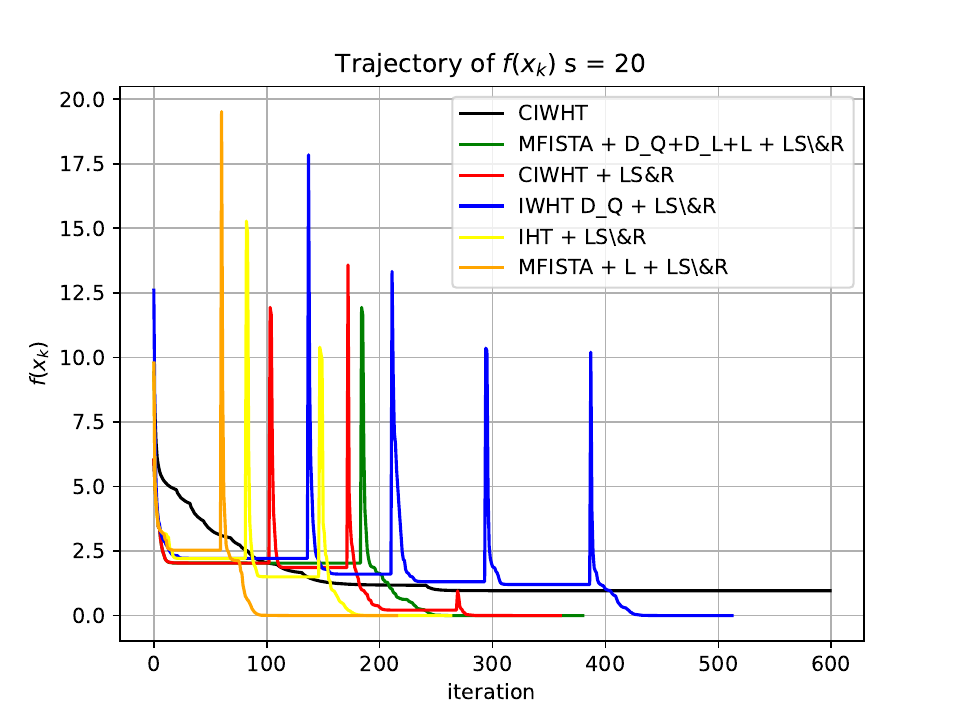}
\includegraphics[width=.32\textwidth]{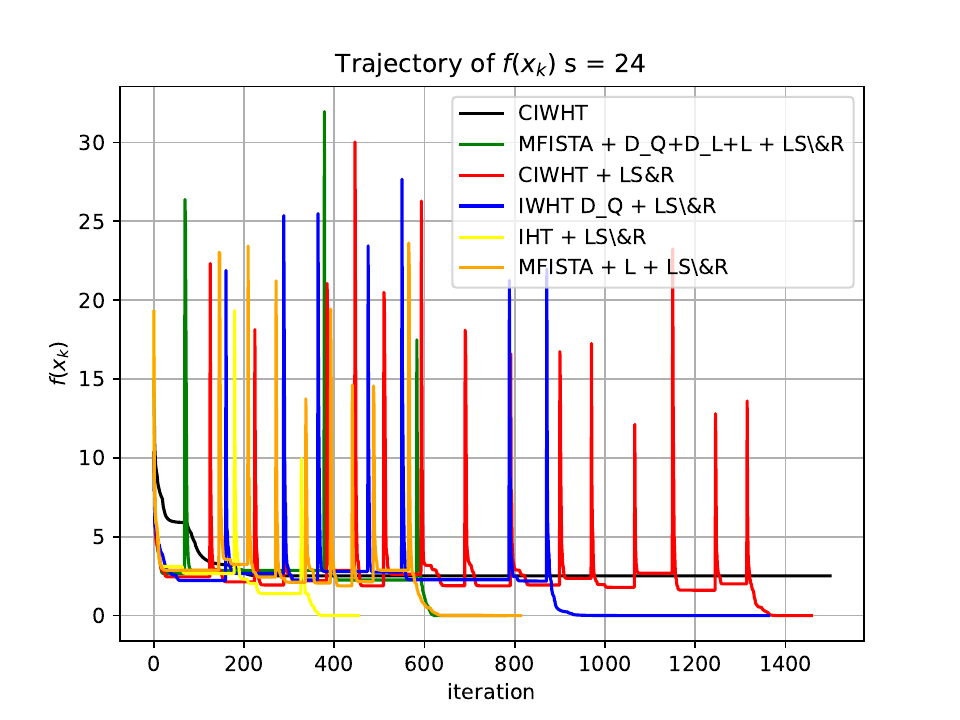}
{Trajectory of $f(x_k)$ for different sparsity levels, $A \in \mathbb{R}^{64 \times 256}$} \\
\includegraphics[width=.32\textwidth]{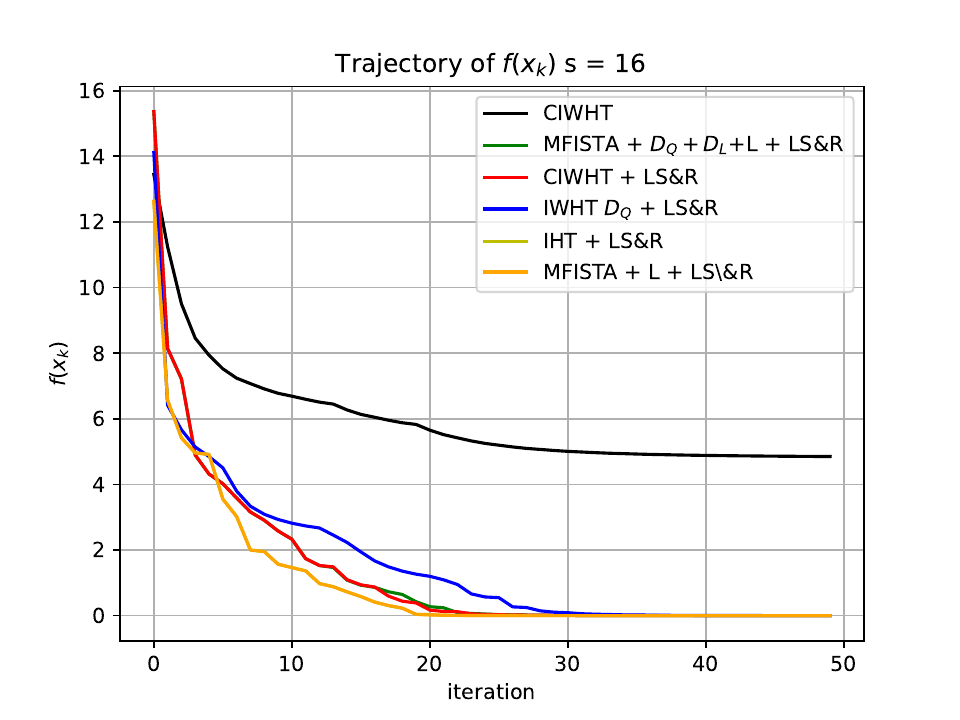}
\includegraphics[width=.32\textwidth]{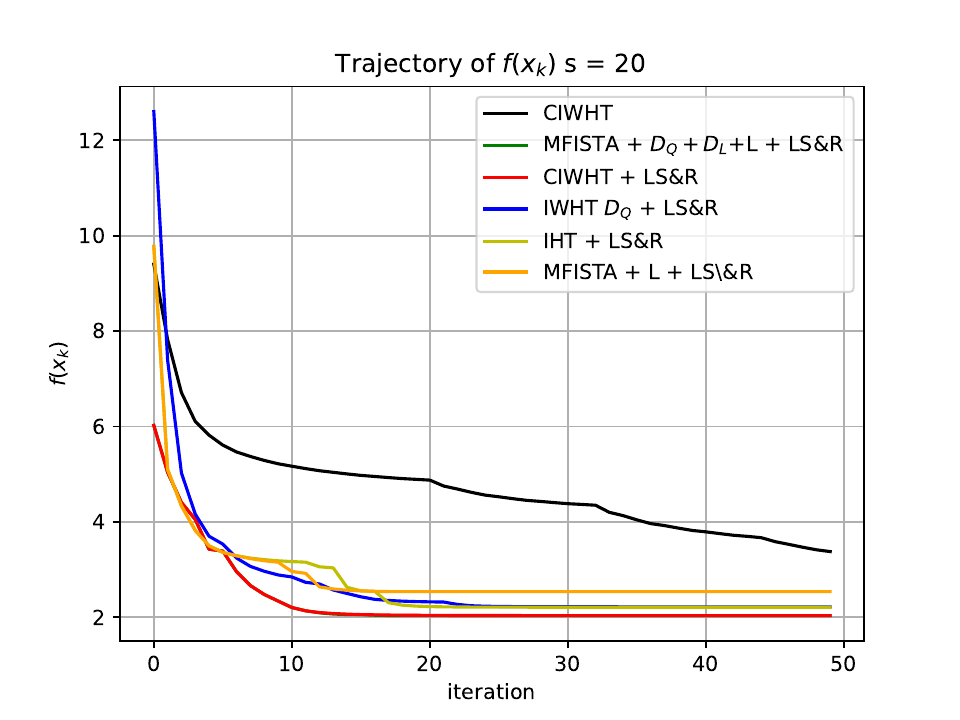}
\includegraphics[width=.32\textwidth]{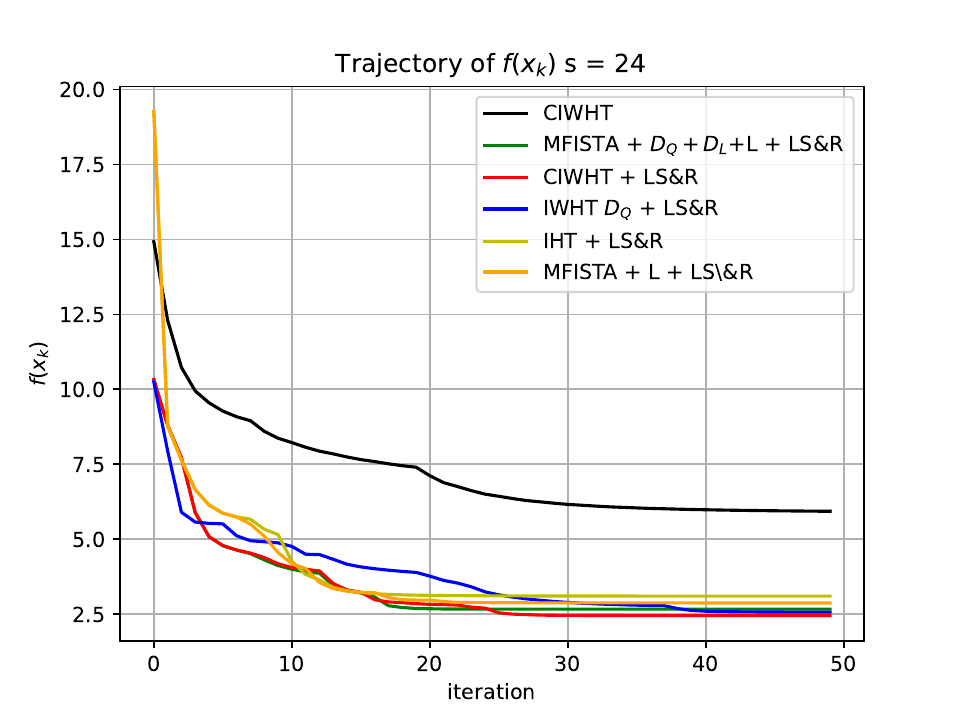}
{Early Trajectory of $f(x_k)$ for different sparsity levels, $A \in \mathbb{R}^{64 \times 256}$}
\caption{Trajectory of $f(x_k)$}
\label{fig:f_x_trajectory}
\end{figure}

\subsection{Computing a DSM}
\label{sec:numerical_compute_weights}
In this section, we show the numerical performance of algorithms described in Section \ref{sec:compute_weights} for computing a DSM. 

Firstly, we compared the performance of BCM algorithms, their parallel versions, and MOSEK \cite{mosek} through CVXPY interface \cite{diamond2016cvxpy}. We generated matrices $A$ of size $n = 250,500,...,5000$ and set $m = \floor*{\frac{n}{8}}$. Then set $C = A^TA$ for the models (\ref{eq:compute_D_L}), (\ref{eq:compute_D_Q}) or in particular their dual counterparts for the numerical algorithms, except for MOSEK, where using the primal model was much faster.  We generate 10 instances for each problem size. Figure \ref{fig:time_dsm} shows the mean CPU and wall time in seconds. 


\begin{figure}[ht]
\centering
\begin{subfigure}{.5\textwidth}
  \centering
  \includegraphics[width=1\linewidth]{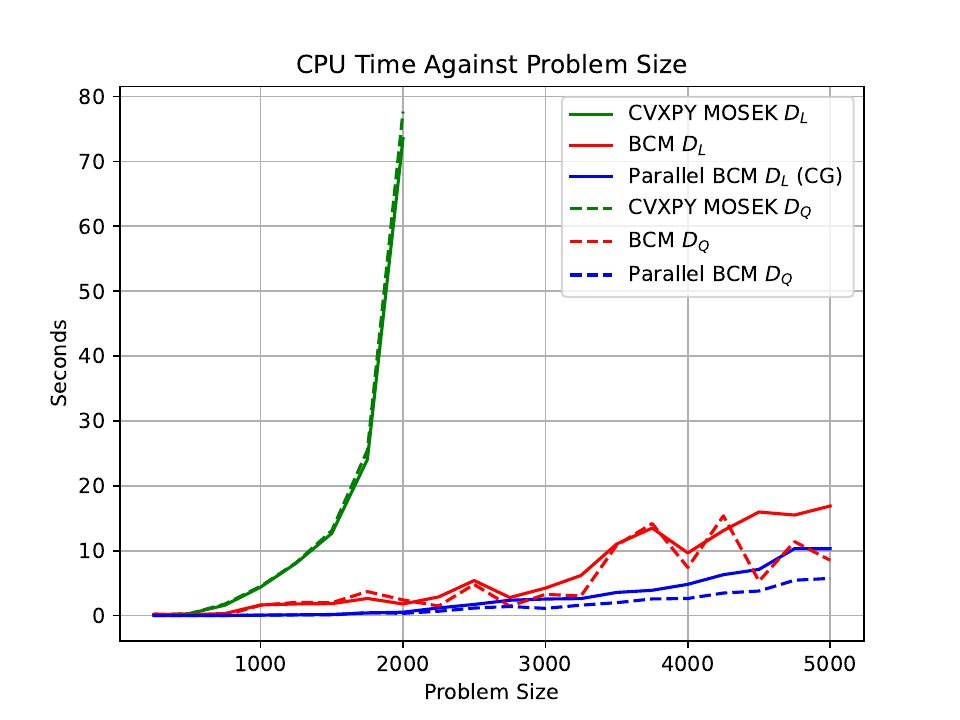}
  \caption{CPU Time in Seconds}
  \label{fig:cpu_dsm}
\end{subfigure}%
\begin{subfigure}{.5\textwidth}
  \centering
  \includegraphics[width=1\linewidth]{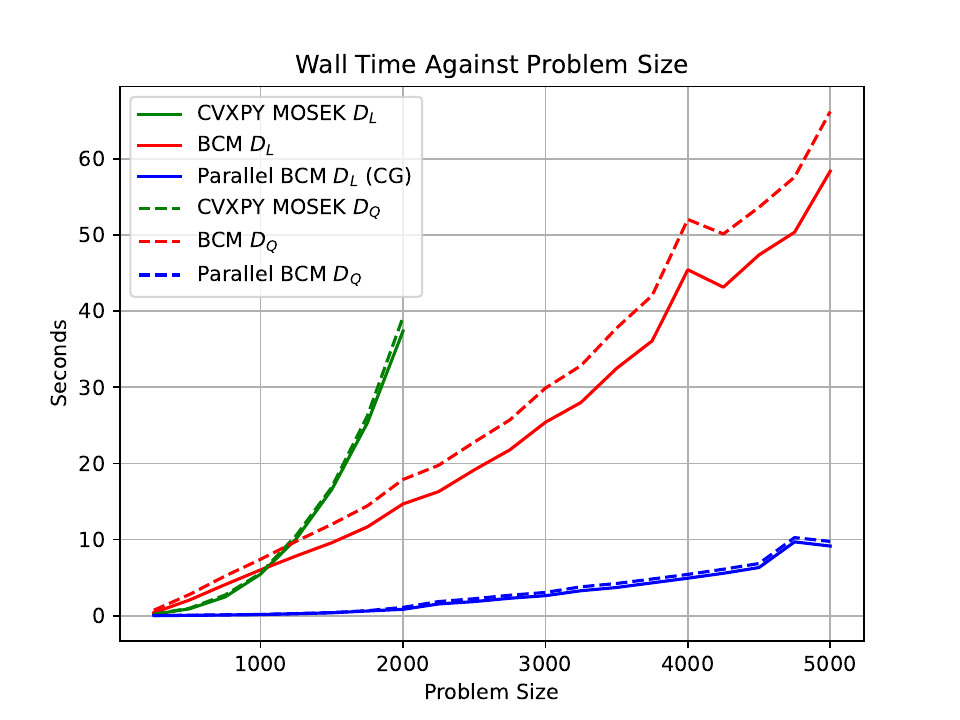}
  \caption{Wall Time in Seconds}
  \label{fig:wall_dsm}
\end{subfigure}
\caption{Cost of computing $D_L$ and $D_Q$ DSMs}
\label{fig:time_dsm}
\end{figure}

Figure \ref{fig:cpu_dsm} shows that the interior point solver is computationally much more demanding while BCM and the parallelized version require much less computational power. Additionally, Figure \ref{fig:wall_dsm} reveals the advantage of the parallelization; while having a similar CPU time, the wall time spent in the parallel version of the BCM algorithm is much less than the original algorithm. Although BCM scales much better than MOSEK for larger problems, the plot suggests that the MOSEK algorithm is faster for smaller problems. This is caused by a fixed number of first-order updates we apply in BCM, whereas interior-point methods check primal-dual conditions for early termination. Ultimately, both figures demonstrate that parallelized versions of the BCM algorithms scale the best in terms of both CPU and wall time. 

Furthermore, Figure \ref{fig:rel_error} illustrates the trajectory of relative error on both primal variable $w$ and dual variable $Z$ for problem sizes $n = 500,1000,2000$ for a single test run. 

\begin{figure}[ht]
\centering
\includegraphics[width=.32\textwidth]{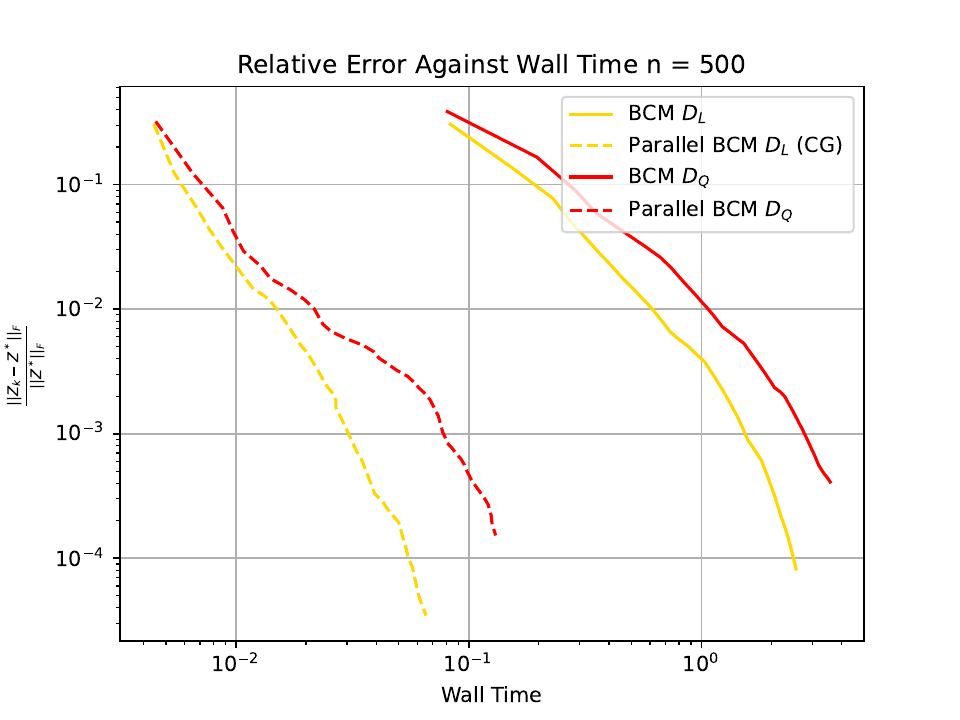}
\includegraphics[width=.32\textwidth]{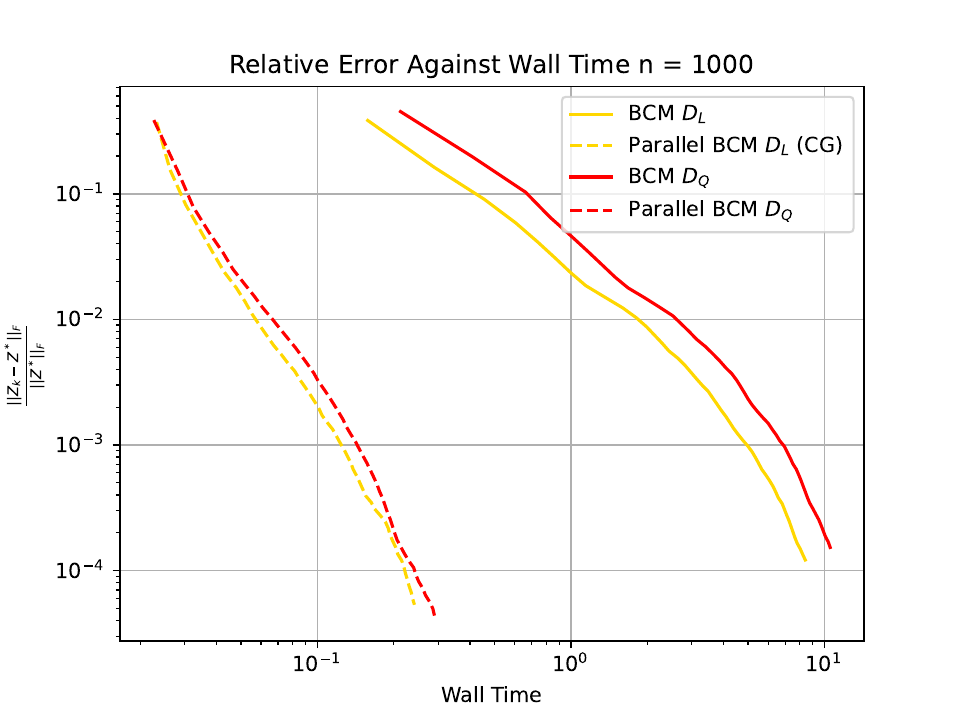}
\includegraphics[width=.32\textwidth]{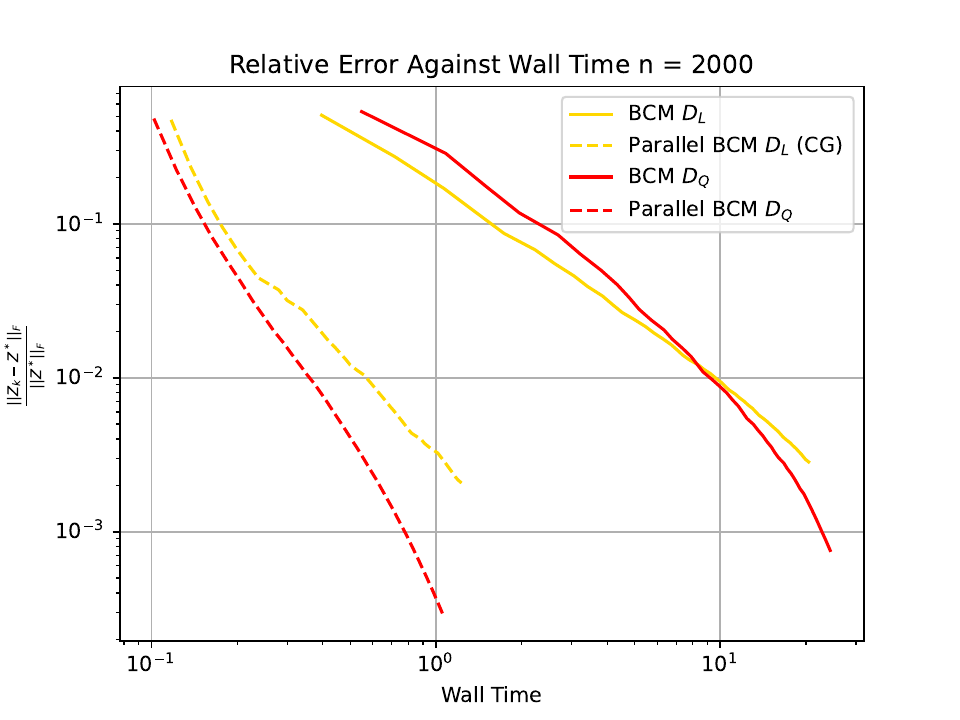}
{Relative Error of the Dual Variable $Z$ For Different Problem Sizes} \\
\includegraphics[width=.32\textwidth]{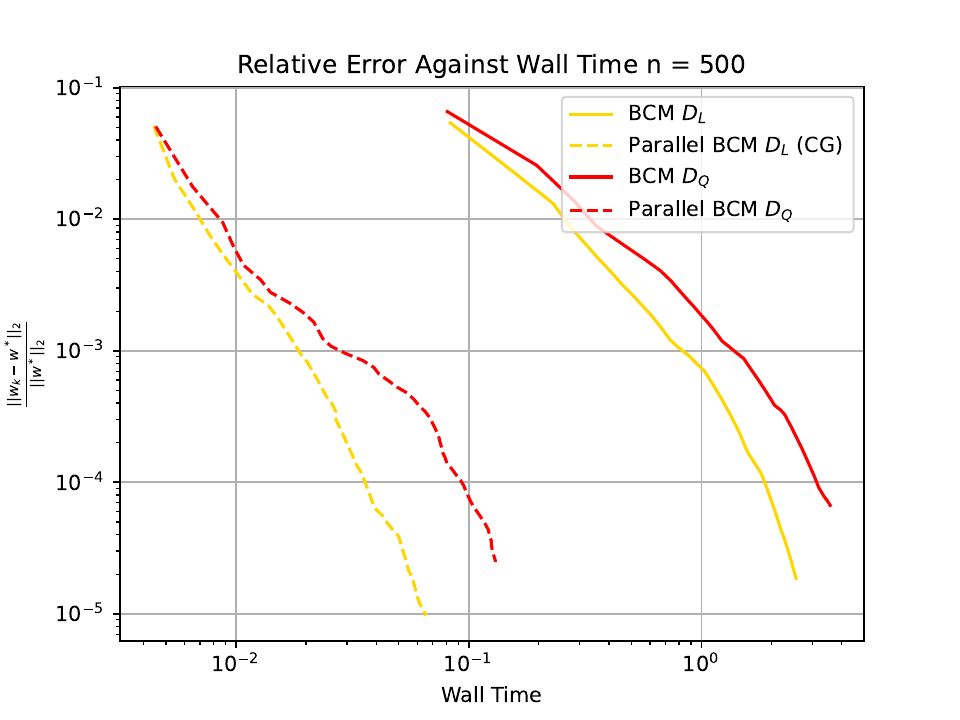}
\includegraphics[width=.32\textwidth]{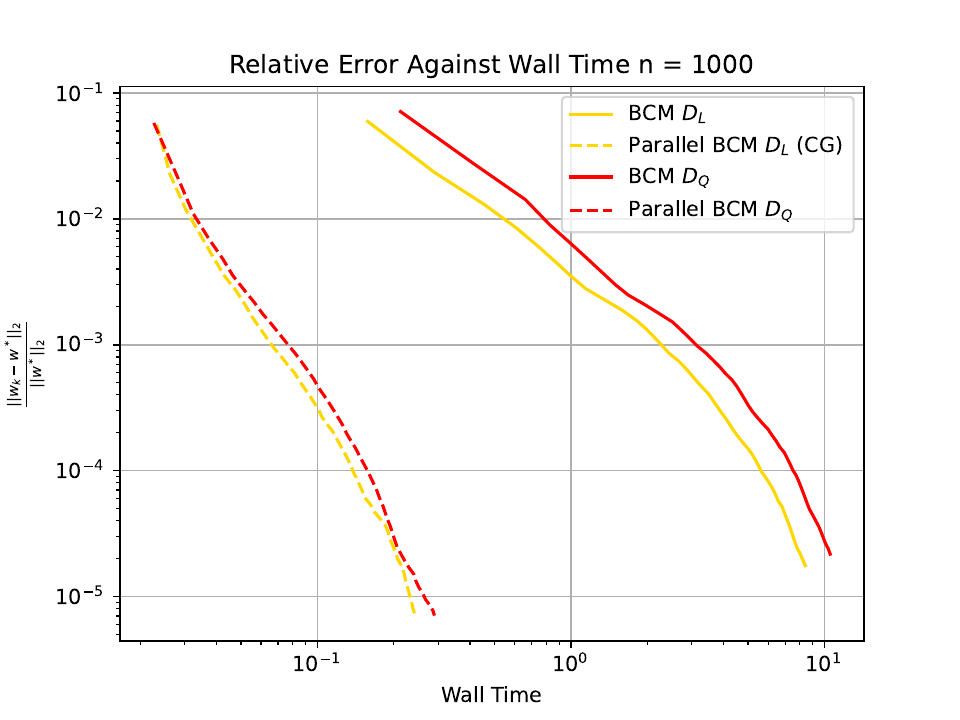}
\includegraphics[width=.32\textwidth]{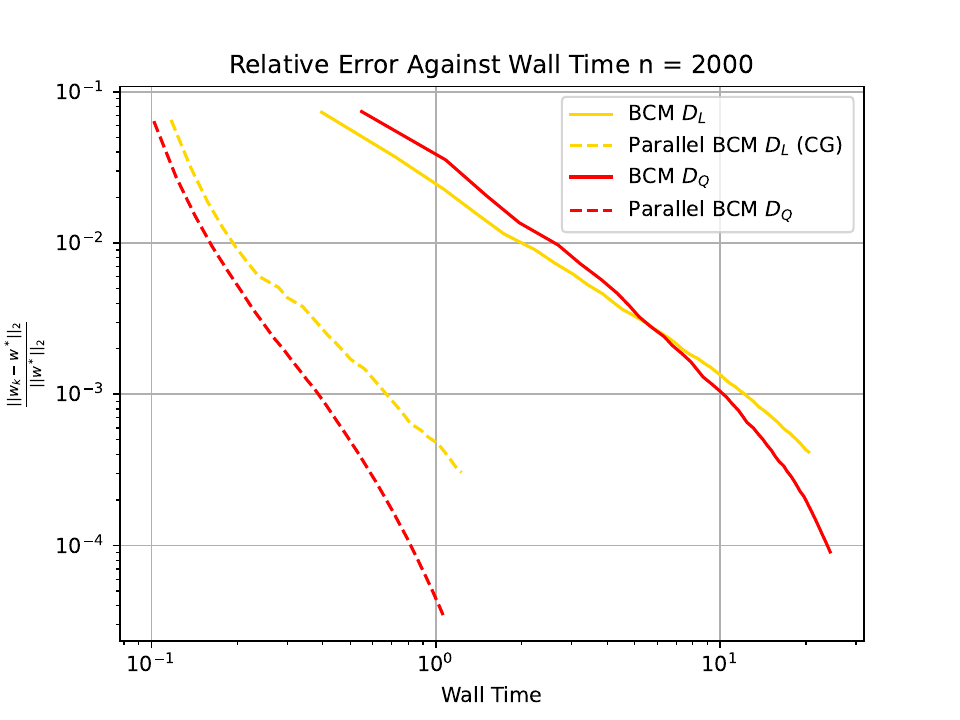}
{Relative Error of the Primal Variable $w$ For Different Problem Sizes}
\caption{Relative error trajectory for various problem sizes}
\label{fig:rel_error}
\end{figure}

All plots in Figure \ref{fig:rel_error} are almost identical, displaying remarkable improvement gained by parallelizing the BCM algorithm. Therefore, we empirically motivate these algorithms for the SDP templates \ref{eq:compute_D_L_dual} and \ref{eq:compute_D_Q_dual}.

\section{Conclusion}

In this paper, we proposed novel extensions to well-known hard-thresholding type algorithms via relative smoothness for sparsity-constrained nonlinear optimization by using Bregman distances. We analyzed the theoretical properties of the IWHT algorithm and proved new convergence results under a mild condition that extends previous results, which shows that IWHT behaves almost like the classical gradient descent or Bregman proximal gradient algorithm. We have drawn a connection between CW-optimality and Bregman stationarity. Then, we proposed a procedure to compute a DSM to equip the novel algorithms. Numerical findings presented in the last section support our theoretical results and show the improvements over other algorithms as a framework. We believe methods developed to construct new algorithms could also be adapted to group sparsity settings to boost the recovery, a topic which we leave for future research.  

\bibliographystyle{siamplain}

\appendix
\section{Convergence Results for Bregman Distance}

In this appendix, we first define the Bregman Proximal Gradient (BPG) algorithm for the sparsity constrained problem then present convergence related results when the Bregman function is separable. Convergence in objective and rate of convergence in the convex case was analyzed in \cite{bregman_descent}, and the non-convex case in \cite{bregman_descent_phase_retrieval}. For the sake of completeness, we describe the generic convergence result obtained by using the Bregman framework analysis. Stationarity analysis and global convergence without using KL-property for the sparsity-constrained case is new. Our order of convergence results follow Section \ref{sec:convergence}. 

We present the Bregman Proximal Gradient algorithm \ref{alg:bpg}, where the inner problem can be solved with (\ref{eq:bregman_descent_seperable_solution}).
\begin{algorithm}[ht]
\caption{Bregman Proximal Gradient Algorithm (BPG)}
\label{alg:bpg}
\begin{algorithmic}[1]
\STATE{\textbf{Input:} $h,L > 0, x_0 \in C_s,s $}
\FOR{$k = 1,2, ...,$}
\STATE{$x_{k+1} \in \underset{x \in C_s}\argmin \ \nabla f(x_k)^T(x-x_k) + LD_h(x,x_{k})$}
\ENDFOR
\end{algorithmic}
\end{algorithm}
To make the Bregman Proximal Gradient step in \ref{alg:bpg} well-defined, we have to make the following assumption:
\begin{assumption}
\label{as:supercoercive_bregman}
$h(x)$ is a supercoercive function. In other words, we have
\begin{equation}
\label{eq:super_coercive}
\begin{array}{ll@{}ll}
\underset{||x|| \to \infty} \lim \dfrac{h(x)}{||x||} = \infty.
\end{array}
\end{equation}
\end{assumption}
This condition ensures a minimizer exists in the BPG step \cite{bregman_descent,bregman_descent_phase_retrieval}. Characterization of the well-posedness of the Bregman Proximal Gradient step is fairly complicated. Instead of dealing with these technical issues, we make this mild assumption that is usually met in practice. 

In addition to the well-defined nature of the algorithm, the supercoercivity assumption on the function $h$ also allows us to express Bregman Stationarity explicitly.
\begin{lemma}
\label{lemma:BPG_stationary_characterization}
Given a strictly convex function $h$, a point $x$ is $L$-Bregman Stationary if and only if
\begin{equation}
\label{eq:BPG_stationary_characterization}
\begin{array}{rl}
\nabla_i f(x) &= 0 \qquad \forall \, i \in I_1(x) \\
\underset{i \in I_1(x)}\min h_i^*(h_i^\prime (x_i)) &\geq \underset{j \in I_0(x)}\max h_j^*( h_j^\prime(0) - \frac{\nabla_j f(x)}{L}),
\end{array}
\end{equation}
where $h_j^*$ is the conjugate function defined by the conjugate (or Fenchel-Legendre) transform
\begin{equation}
\label{eq:fenchel_conjugate}
\begin{array}{ll@{}ll}
h^*(y) = \underset{x}\sup \ y^Tx - h(x).
\end{array}
\end{equation}
\end{lemma}

\begin{proof}
Using the definition of Bregman Stationarity (\ref{eq:bregman_stationarity}) we get the following optimality condition for non-zero variables;
\begin{equation}
\label{eq:BPG_stationary_characterization1}
\begin{array}{ccc}
\nabla_i f(x) + Lh_i^\prime(x_i) - L h_i^\prime(x_i) = 0 &\forall i \in I_1(x) \\
\implies \nabla_i f(x) = 0 &\forall i \in I_1(x).
\end{array}
\end{equation}
This proves the first condition in (\ref{eq:BPG_stationary_characterization}). For the second condition, we need to use the separability of the problem and calculate the minimizer for each variable. Then to find the best sparse solution, we greedily compare the separable objective at the optimal value and the value at zero. Let $i \in I_1(x)$. Calculating the objective at zero and optimal value gives
\begin{equation}
\label{eq:BPG_stationary_characterization2}
\begin{array}{ccc}
\underset{y_i}\min \ \phi_i(y_i) = (\nabla_i f(x) - L h_i^\prime(x_i))(y_i-x_i) + Lh_i(y_i) - Lh_i(x_i) \\
\nabla_i f(x)  - L h_i^\prime(x_i) + L h_i^\prime(x_i) = 0 \\
y_i^* = x_i \\
\phi_i(y_i^*) = 0, \ \phi_i(0) = L h_i^\prime(x_i)x_i- L h_i(x_i) + Lh_i(0) = Lh_i^*(h_i^\prime(x_i)) + Lh_i(0).
\end{array}
\end{equation}
Going from the second to the third line, we used $\nabla_i f(x) = 0$ from (\ref{eq:BPG_stationary_characterization1}). In the last line, we used the conjugate formula. Making similar calculations for $j \in I_0(x)$ yields
\begin{equation}
\label{eq:BPG_stationary_characterization3}
\begin{array}{ccc}
\underset{y_j}\min \ \phi_j(y_j) = (\nabla_j f(x) - L h_j^\prime(x_j))y_j + Lh_j(y_j) - Lh_j(0) \\
\phi_j(y_j^*) = -L h_j^*(-\frac{\nabla_j f(x)}{L} + h_j^\prime(0)) - L h_j(0), \ \phi_j(0) = 0 .\\
\end{array}
\end{equation}
Since $x$ is $L$-Bregman Stationary, substituting a non-zero variable $i \in I_1(x)$ for a zero variable $j \in I_0(x)$ should not give a decrease in the extended descent equation (\ref{eq:bregman_descent_lemma}). Thus;
\begin{equation}
\label{eq:BPG_stationary_characterization4}
\begin{array}{rll}
-L h_j^*(-\frac{\nabla_j f(x)}{L} + h_j^\prime(0)) &+ Lh_i^*(h_i^\prime(x_i)) \geq 0, \forall i \in I_1(x), j \in I_0(x) \\
h_i^*(h_i^\prime(x_i)) &\geq h_j^*(-\frac{\nabla_j f(x)}{L} + h_j^\prime(0)), \forall i \in I_1(x), j \in I_0(x) \\
\underset{i \in I_1(x)}\min h_i^*(h_i^\prime(x_i)) &\geq \underset{j \in I_0(x)}\max h_j^*(-\frac{\nabla_j f(x)}{L} + h_j^\prime(0)).
\end{array}
\end{equation}
In the first line, to simplify the calculations a little bit, we use a convention that $h_i(0) = 0$. If this is not true, we can simply subtract $h_i(0)$ from the function which in return adds $h_i(0)$ to the conjugate. An alternative assumption to get the final nice expression would be assuming all $h_i$'s to be the same function. 

We note that the interchange between $\inf$ and $\min$ in conjugate function calculations above is justified by Assumption \ref{as:supercoercive_bregman}, which guarantees that a minimizer exists. 
\end{proof}

\begin{theorem}
\label{thm:BPG_convergence_subsequential}
Let $(x_k)$ be a sequence generated by the BPG Algorithm (Bregman Hard Thresholding) with a separable Bregman distance $h = \sum_{i=1}^n h_i(x_i)$. Let $L_hh-f$ be convex and $L_h < L$. Then, the following holds:
\begin{enumerate}
    \item $f(x_{k+1}) \leq  f(x_k) - (L - L_h) D_h(x_{k+1},x_{k})$. 
    \item $\sum_{i=0}^{\infty} D_h(x_{k+1},x_k) < \infty$, and hence $\lim_{k\to \infty} D_h(x_{k+1},x_k) = 0$.
    \item Minimum stationarity gap encountered decreases with rate $O(\frac{1}{n})$.
\end{enumerate}
\end{theorem}

\begin{proof}
As in the proof of Lemma \ref{lemma:subsequential_convergence}, all we need to show is the descent inequality for the Bregman distance which is very similar to (\ref{eq:iwht_descent_lemma})
\begin{equation}
\label{eq:bpg_descent_lemma}
\begin{array}{ll@{}ll}
f(x_{k+1}) + (L - L_h) D_h(x_{k+1},x_k) \! & \!\leq f(x_k) + \nabla f(x_k)^T(x_{k+1}-x_k) + LD_h(x_{k+1},x_k) \\ 
&\!\leq f(x_k).
\end{array}
\end{equation}
The rest of the statements follow from the first one.
\end{proof}

By assuming s-sparse coercivity (\ref{eq:sparse_coercive}), we guarantee the boundedness of the sequence $(x_k)$ and the existence of limit points. An analog of  Theorem \ref{lemma:limit_d_stationary} when using Bregman distance is as follows:

\begin{theorem}
\label{lemma:BPG_convergence_subsequential_limit}
In the setting of Lemma \ref{thm:BPG_convergence_subsequential}, any limit point of the sequence $(x_k)$ is $L$-Bregman Stationary. 
\end{theorem}
\begin{proof}
Assume to the contrary, a limit point $\Bar{x}$ of the sequence $(x_k)$ is not $L$-Bregman Stationary. Then $\exists y \in C_s$ such that
\begin{equation}
\label{eq:bpg_bregman_stationary}
\begin{array}{ll@{}ll}
\nabla f(\Bar{x})^T(y-\Bar{x}) + LD_h(y,\Bar{x}) = - \delta < 0.
\end{array}
\end{equation}
Let $x_{n_k}$ be a subsequence that converges to $\Bar{x}$. As soon as $(x_{n_k})$ is sufficiently close to $\Bar{x}$, we preserve the inequality with $\delta/2$. Then this creates a contradiction to convergence in objective by using an argument similar to (\ref{eq:iwht_L_stationary2}).
\end{proof}

\begin{assumption}
For the point convergence, we will need the following assumptions:
\begin{enumerate}
    \item If $(x_k)$ converges to some $x$ in $\dom \textit{h}$ then $D_h(x,x_k) \to 0$.
    \item Conversely, if $D_h(x,x_k) \to 0$ for some $x \in \dom \textit{h}$, then $x_k \to x$.
    \item Further, if $D_h(x_{k+1},x_k) \to 0$, then $||x_{k+1}-x_k|| \to 0$.
\end{enumerate}
\end{assumption}

The first two assumptions are used in convergence analysis in \cite{bregman_descent}. We assume a further condition to be able to identify the correct sparsity pattern when a point is sufficiently close to some point. For instance, the third condition is violated for Burg entropy $-log(x)$ with $x_k = ak, \exists \, a > 0, k \geq 1$. These conditions are easily satisfied under a strong convexity assumption on $h(x)$. 

Finally, we have the convergence of the entire sequence for the BPG algorithm as well. 

\begin{theorem}
\label{thm:bpg_convergence}
Let $(x_k)$ be a sequence generated by the BPG Algorithm with a separable Bregman distance $h(x) = \sum_{i=1}^n h_i(x_i)$. Let $L_hh-f$ be convex and $L_h < L$. If there exists a limit point $\Bar{x}$ of $(x_k)$, where  $||\Bar{x}||_0 = s$, then the entire sequence converges to $\Bar{x}$.
\end{theorem}

\begin{proof}
\label{thm:BPG_convergence_proof}
Firstly, we need a slightly stronger descent inequality to show convergence than (\ref{eq:bpg_descent_lemma}).   Lemma \ref{lemma:three_points_descent} achieves this by using the three-points identity \cite{chen_teboulle1993}.
\begin{lemma}[Descent inequality for the Bregman distance \cite{bregman_descent,tseng_convex_concave}]
\label{lemma:three_points_descent}
Let $x^+$ be defined as follows
\begin{equation}
\label{eq:bregman_minimization}
\begin{array}{ll@{}ll}
x^+ \in \underset{y}\argmin \ \nabla f(x)^T(y-x) + LD_h(y,x).
\end{array}
\end{equation}
Then the following inequality holds $\forall u \in  \dom(h)$
\begin{equation}
\label{eq:bregman_minimization1}
\begin{array}{ll@{}ll}
\nabla f(x)^T(x^+ - x) + LD_h(x^+,x) + LD_h(u,x^+) \leq \nabla f(x)^T(u - x) + LD_h(u,x).
\end{array}
\end{equation}
Furthermore, as a consequence of the previous inequality, we can derive the following descent inequality 
\begin{equation}
\label{eq:bregman_minimization2}
\begin{array}{ll@{}ll}
f(x^+)-f(u) \leq LD_h(u,x) - LD_h(u,x^+) - (L-L_h)D_h(x^+,x) \ \forall u \in \dom(h).
\end{array}
\end{equation}
\end{lemma}
Let $(x_{n_k})$ be a subsequence that converges to $\Bar{x}$. For large $k$, we have $||x_{n_k}-\Bar{x}|| < \epsilon_1$, and by using Assumption 2 part 1 we also have $D_h(\Bar{x},x_{n_k}) < \epsilon_2$. Furthermore, by Assumption part 3 and third result in Theorem \ref{thm:BPG_convergence_subsequential}, we also have $||x_{k+1}-x_{k}||_2 < \epsilon_3$. Then, for sufficiently small chosen $\epsilon_1$ and $\epsilon_3$, we have that the sparsity pattern won't change in the next iteration. Therefore, the algorithm will apply a Bregman Proximal Gradient algorithm in the same sparsity pattern which allows us to utilize (\ref{eq:bregman_minimization1}). Define the following quantities as in (\ref{eq:D_convergence1});
\begin{equation}
\label{eq:bregman_descent_inequality2}
\begin{array}{ll@{}ll}
I_1(x_{n_k+1}) = I_1(x_{n_k}) = I_1(\Bar{x}) = S \\
z_{n_k} = (x_{n_k})_{S} \\
z_{n_k+1} = (x_{n_k+1})_{S} \\
\Bar{z} = (\Bar{x})_{S} \\ 
\Bar{f}(z) = f(x), \quad x_S = z, x_{S^C} = 0 \\
\Bar{h}(x) = \underset{i \in S}\sum h_i(x)
\end{array}
\end{equation}
Then applying the previous inequality with $u = \Bar{z}$ gives
\begin{equation}
\label{eq:bregman_descent_inequality3}
\begin{array}{rlll}
D_{\Bar{h}}(\Bar{z},z_{n_k+1})  \leq D_{\Bar{h}}(\Bar{z},z_{n_k}) - (1 - \frac{L_h}{L})D_{\Bar{h}}(z_{n_k+1},z_{n_k}) - \frac{1}{L}(\Bar{f}(z_{n_k+1})-\Bar{f}(\Bar{z})).
\end{array}
\end{equation}
Since Bregman distance is non-negative and $f(x_k)$ is non-increasing, we see that in the next iteration $z_{n_k+1}$ is closer to $\Bar{z}$ than $z_{n_k}$, in terms of Bregman distance. This implies $D_h(\Bar{x},x_{n_k+1}) \leq D_h(\Bar{x},x_{n_k})$ since the Bregman distance has a separable form. By using assumption 2, we can guarantee that $||x_{n_k}-\Bar{x}|| < \epsilon_4$, which implies that in the next iteration, (\ref{eq:bregman_descent_inequality3}) will hold again by the same argument. Then, $D_h(\Bar{x},x_k)$ decreases down to zero. Combining this result and applying assumption 2 again, we conclude that sequence $(x_k)$ converges to $\Bar{x}$.

If $||\Bar{x}||_0 < s$, then by using Bregman Stationarity property from \ref{lemma:BPG_convergence_subsequential_limit} and inspecting the optimality condition in (\ref{eq:bregman_stationarity}), we have $\nabla f(\Bar{x}) = 0$ as in the proof of \ref{thm:L_convergence}. As in the previous case for large $k$, we have $||x_{n_k}-\Bar{x}||_2 < \epsilon_1$, $D_h(\Bar{x},x_{n_k}) < \epsilon_2$ for any choice of $\epsilon_1,\epsilon_2 > 0$. With a similar line of reasoning as before we also have  $||x_{k+1}-x_{k}||_2 < \epsilon_3$. Then, from $||x_{n_k}-\Bar{x}||_2 < \epsilon_1$ we get $I_1(x_{n_k}) \supset I_1(\Bar{x})$. By using $||x_{k+1}-x_{k}||_2 < \epsilon_3$, we can also argue that $I_1(x_{n_k+1}) \supset I_1(\Bar{x})$. Then we have two cases. If $I_1(x_{n_k}) = I_1(x_{n_k+1})$ then the previous arguments follow and therefore inequality in the (\ref{eq:bregman_descent_inequality3}) holds, and we observe monotonic decrease in the Bregman distance $D_(\Bar{x},x_{k})$ for large $k$, and we have convergence. If $I_1(x_{n_k}) \neq I_1(x_{n_k+1})$, define the sparsity patterns $S_1 = I_1(x_{n_k+1}) \cap I_1(x_{n_k}), S_2 = I_1(x_{n_k}) \setminus  I_1(x_{n_k+1}), S_3 = I_1(x_{n_k+1}) \setminus I_1(x_{n_k}), S = S_1 \cup S_2 \cup S_3 $. Now let us apply the first version of the descent inequality (\ref{eq:bregman_minimization1}) to the space of $S_1 \cup S_3$.
\begin{equation}
\label{eq:bregman_descent_inequality4}
\begin{array}{ll@{}ll}
\langle
\begin{bmatrix}
\nabla_{S_1} \Bar{f}(z_{n_k}) \\ 
0 \\
\nabla_{S_3} \Bar{f}(z_{n_k})  \\
\end{bmatrix}
,
\begin{bmatrix}
(z_{n_k+1})_{S_1}- (z_{n_k})_{S_1}   \\
0 \\
(z_{n_k+1})_{S_3} \\
\end{bmatrix}

\rangle

+L
D_{\Bar{h}}\left(
\begin{bmatrix}
(z_{n_k+1})_{S_1} \\
0 \\
(z_{n_k+1})_{S_3} \\
\end{bmatrix}
,
\begin{bmatrix}
 (z_{n_k})_{S_1} \\
0 \\
0  \\
\end{bmatrix}
\right)

\leq \\

\langle
\begin{bmatrix}
\nabla_{S_1} \Bar{f}(z_{n_k}) \\ 
0 \\ 
\nabla_{S_3} \Bar{f}(z_{n_k})  \\
\end{bmatrix}
,
\begin{bmatrix}
(\Bar{z})_{S_1}- (z_{n_k})_{S_1}   \\
0 \\
0 \\
\end{bmatrix}
\rangle

+L
D_{\Bar{h}}\left(
\begin{bmatrix}
 (\Bar{z})_{S_1} \\
0 \\
0  \\
\end{bmatrix}
,
\begin{bmatrix}
(z_{n_k})_{S_1} \\
0 \\
0 \\
\end{bmatrix}
\right)

\\

-L
D_{\Bar{h}}\left(
\begin{bmatrix}
 (\Bar{z})_{S_1} \\
0 \\
0  \\
\end{bmatrix}
,
\begin{bmatrix}
(z_{n_k+1})_{S_1} \\
0 \\
(z_{n_k+1})_{S_3} \\
\end{bmatrix}
\right).
\end{array}
\end{equation}

Note that there is a lot of zero padding in the inequalities which we did not need to write. We will add additional terms to fill in the zeros and complete the inequality to a more familiar form. To that end, let us add $\nabla_{S_2} \Bar{f}(z_{n_k})^T(-z_{n_k}) + LD_{\Bar{h}}([0,0,0]^T,[0,(z_{n_k})_{S_2},0]^T)$ to the both sides of the inequality. Then we can compactly write the equation as;
\begin{equation}
\label{eq:bregman_descent_inequality5}
\begin{array}{ll@{}ll}
\nabla \Bar{f}(z_{n_k})^T(z_{n_k+1}-z_{n_k}) + LD_{\Bar{h}}(z_{n_k+1},z_{n_k}) &\leq \nabla \Bar{f}(z_{n_k})^T(\Bar{z}-z_{n_k}) + LD_{\Bar{h}}(\Bar{z},z_{n_k}) \\
&- LD_{\Bar{h}}(\Bar{z},z_{n_k+1}).
\end{array}
\end{equation}
To complete the proof we need to start from the weaker descent inequality initially proved in Theorem \ref{thm:BPG_convergence_subsequential}. Starting with equation (\ref{eq:bpg_descent_lemma})
\begin{equation}
\label{eq:bregman_descent_inequality6}
\begin{array}{ll@{}ll}
f(x_{n_k+1}) + (L-L_h)D_h(x_{n_k+1},x_{n_k}) &\leq f(x_{n_k}) + \nabla f(x_{n_k})^T(x_{n_k+1}- x_{n_k}) \\
&+LD_h(x_{n_k+1},x_{n_k}) \\
&\leq f(x_{n_k}) + \nabla f(x_{n_k})^T(\Bar{x}-x_{n_k}) \\
&+ LD_h(\Bar{x},x_{n_k}) - L D_h(\Bar{x},x_{n_k+1}) \\
&\leq f(\Bar{x}) + LD_h(\Bar{x},x_{n_k}) - L D_h(\Bar{x},x_{n_k+1}).
\end{array}
\end{equation}
In the first inequality, we switched from the $z$ domain to the $x$ domain since they only have non-zero values on $S$. In the second inequality, we used the inequality derived in the (\ref{eq:bregman_descent_inequality5}). The last inequality follows from the convexity of $f$. Rearranging the terms we arrive at the same inequality as in (\ref{eq:bregman_descent_inequality3}). Repeating the same arguments proves the convergence in this case as well. 
\end{proof}

We remark that as a byproduct of the above computation, we have also proved (\ref{eq:bregman_minimization2}) starting from (\ref{eq:bregman_minimization1}).

Lastly, we can argue for the convergence rate of the BPG algorithm by utilizing the convergence rate in the convex case.

\begin{corollary}
\label{cor:convergence_rate_bregman}
Let $(x_k)$ be a sequence generated by the Bregman Proximal Gradient algorithm. Assume that the sequence $(x_k)$ is bounded. Then for sufficiently large k, the function value converges with $O(1/k)$.
\end{corollary}
\begin{proof}
As in the proof of \ref{cor:convergence_rate}, combining the point convergence proof from Theorem \ref{thm:bpg_convergence} with the convergence rate obtained for objective value in  \cite{bregman_descent} gives the desired result. 
\end{proof}

We note that in both corollaries \ref{cor:convergence_rate} and \ref{cor:convergence_rate_bregman}, as long as the sparsity pattern remains the same, $O(1/k)$ decrease rate in the objective is also true but the objective does not decrease to the final value with $O(1/k)$ but to a higher lower bound. In other words, the objective value plot looks like a monotone decreasing piecewise function where the pieces are described by a function in the form of $g(x) = a/x + b$ for some $a,b \in \mathbb{R}$. 

\section{Proof of An Additional Inequality for Linear Convergence Under Strong Convexity}
\label{ap:inequality_linear_conv_strong_conv}
Finally, we present a derivation of an inequality that is useful for proving linear convergence under restricted strong convexity. In the good case when we have convergence to a $s$-sparse solution, the proof follows from the classical convergence argument. However, the limit point may have cardinality less than $s$. In this case, we need to argue separately since there could be infinitely often sparsity pattern changes which hinders the proof. We will derive the inequality for a function $f$ that is smooth with respect to $\frac{1}{2}x^THx$, and IWHT uses some diagonal matrix $D \succ H$, and also $f$ is $s$-RSC with respect to $\frac{1}{2}x^TDx$ i.e. the following inequality holds for some $\sigma$: 
\begin{equation}
\label{eq:restricted_D_strong_convexity}
\begin{array}{ll@{}ll}
f(y) \geq f(x) + \nabla f(x)^T(y-x) + \frac{\sigma}{2}||x-y||_D^2 \quad  \forall x,y: ||x-y||_0 \leq s.
\end{array}
\end{equation}
Since this function satisfies the growth condition, the sequence generated by the IWHT algorithm will have a limit point $\Bar{x}$, and by \ref{thm:L_convergence} the entire sequence will converge to it. Assume that the limit point has cardinality less than $s$. Let $K$ be large so that $(x_k)$ is sufficiently close to $\Bar{x}$ so that $I_1(x_k) \cap I_1(\Bar{x}) $ for all $k \geq K$. If $I_1(x_{k+1}) = I_1(x_{k})$ then classical inequalities used in the proof of linear convergence hold. Assume $I_1(x_{k+1}) \neq I_1(x_k)$. Let us also define sets $S_1 = I_1(x_{k+1}) \cap I_1(x_{k}), S_2 = I_1(x_{k}) \setminus I_1(x_{k+1}), S_3 = I_1(x_{k+1}) \setminus I_1(x_{k})$. Then starting with the descent lemma and writing terms explicitly gives
\begin{equation}
\label{eq:linear_convergence_inequality1}
\begin{array}{ll@{}ll}
f(x_{k+1}) + &\frac{1}{2}||x_{k+1}-x_k||_{D-H}^2 \leq f(x_k)+ \nabla f(x_k)^T(x_{k+1}-x_k) + \frac{1}{2}||x_{k+1}-x_k||_D^2 \\
f(x_{k+1})  & \leq f(x_k) - \frac{1}{2}||\nabla_{S_1} f(x_k)||_{D_1^{-1}}^2 - \frac{1}{2}||\nabla_{S_3} f(x_k)||_{D_3^{-1}}^2 - \nabla_{S_2} f(x_k)^T(x_k)_{S_2} \\
& + \frac{1}{2}||(x_k)_{S_2}||_{D_2}^2 \\
f(x_{k+1})  & \leq f(x_k) - \frac{1}{2}||\nabla_{S_1} f(x_k)||_{D_1^{-1}}^2 - \frac{1}{2}||\nabla_{S_3} f(x_k)||_{D_3^{-1}}^2 - \frac{1}{2}||\nabla_{S_2} f(x_k)||_{D_2^{-1}}^2 \\
& + \frac{1}{2}||(x_k)_{S_2} - D_{S_2}^{-1}\nabla_{S_2} f(x_k)||_{D_2}^2 \\
f(x_{k+1})  & \leq f(x_k) - \frac{1}{2}||\nabla_{S_1} f(x_k)||_{D_1^{-1}}^2 - \frac{1}{2}||\nabla_{S_2} f(x_k)||_{D_2^{-1}}^2.
\end{array}
\end{equation}
The first line follows from the descent lemma. In the second line, we drop the $\frac{1}{2}||x_{k+1}-x_k||_{D-H}^2$ term and write the $x_{k+1}$ explicitly.  In the third line, we complete the squares for $S_2$ terms. In the last line, we use the fact that $||\nabla_{S_3} f(x_k)||_{D_3^{-1}}^2 \geq ||(x_k)_{S_2} - D_{S_2}^{-1}\nabla_{S_2} f(x_k)||_{D_2}^2$ which follows from the minimization problem given in (\ref{eq:D_descent}). 

The second key inequality we will need follows by substituting $y = \Bar{x}, x = x_k$ in (\ref{eq:restricted_D_strong_convexity})
\begin{equation}
\label{eq:linear_convergence_inequality2}
\begin{array}{lllllll}
&f(\Bar{x}) & \geq f(x_k)+ \nabla f(x_k)^T(\Bar{x}-x_k) + \frac{\sigma}{2}||\Bar{x}-x_k||_D^2 \\
&f(\Bar{x})  & \geq f(x_k) + \nabla_{S_1} f(x_k)^T(\Bar{x}-x_k)_{S_1} + \frac{\sigma}{2}||(\Bar{x}-x_k)_{S_1}||_{D_1}^2 + \frac{\sigma}{2}||(x_k)_{S_2}||_{D_2}^2 \\
&& - \nabla_{S_2} f(x_k)^T(x_k)_{S_2}  \\
&f(\Bar{x})  & \geq f(x_k) - \frac{1}{2\sigma}||\nabla_{S_1} f(x_k)||_{D_1^{-1}}^2 - \frac{1}{2\sigma}||\nabla_{S_2} f(x_k)||_{D_2^{-1}}^2 \\
&&+\frac{1}{2}||(x_k)_{S_2} - D_{S_2}^{-1}\nabla_{S_2}f(x_k)||_{D_2} \\
&&\frac{1}{2\sigma}||\nabla_{S_1} f(x_k)||_{D_1^{-1}}^2 + \frac{1}{2\sigma}||\nabla_{S_2} f(x_k)||_{D_2^{-1}}^2\\
&&\geq f(x_k) - f(\Bar{x}) + \frac{\sigma}{2}||(x_k)_{S_2} - \sigma^{-1}D_{S_2}^{-1}\nabla_{S_2}f(x_k)||_{D_2}.
\end{array}
\end{equation}
The first line is simply substitution. In the second line, some terms are discarded by exploiting the sparsity patterns. In the third line, we take the minimum over $\Bar{x}$ and also complete the squares for $S_2$ terms. In the last line gradient terms are rearranged. We can further drop the  $\frac{\sigma}{2}||(x_k)_{S_2} - \sigma^{-1}D_{S_2}^{-1}\nabla_{S_2}f(x_k)||_{D_2}$ term on the right-hand side as we do not need that term for the proof. Combining the last inequalities in (\ref{eq:linear_convergence_inequality1}) and (\ref{eq:linear_convergence_inequality2}) gives;
\begin{equation}
\label{eq:linear_convergence_inequality3}
\begin{array}{rll}
f(x_k)- f(x_{k+1}) &\geq \sigma(f(x_k)-f(\Bar{x})) \\
f(x_k)(1-\sigma) - \sigma f(\Bar{x}) &\geq f(x_{k+1}) \\
(1-\sigma)(f(x_k)-f(x^*)) &\geq f(x_{k+1}) - f(x^*).
\end{array}
\end{equation}
From this inequality, the linear convergence in the $f(x_k)$ sequence is clear. 

\section{Extracting Primal Variable From Dual Solution}
\label{ap:get_primal}
In this section, we briefly go over how to compute primal variables from dual optimal solutions.

\begin{lemma}
Let $Z$ be an optimal solution to (\ref{eq:compute_D_L_dual}). Then a corresponding primal optimal solution can be computed by;
\begin{equation}
\label{eq:primal_solution_L}
\begin{array}{ll@{}ll}
w_i = \frac{Z_i^T Z C_i}{Z_i^TZ_i} \quad \forall i = 1,..,n.
\end{array}
\end{equation}
\end{lemma}
\begin{proof}
Writing the Complementary Slackness condition corresponding to Primal-Dual pair models (\ref{eq:compute_D_L}) and (\ref{eq:compute_D_L_dual}) gives \cite{yinyu_lecture_notes};
\begin{equation}
\label{eq:comp_slack_L}
\begin{array}{ll@{}ll}
&\tr(Z(\Diag(w)-C)) = 0 \\
\iff & Z(\Diag(w)-C) = 0 \\
&Z\Diag(w) = ZC \\
&w_i = \frac{Z_i^T Z C_i}{Z_i^TZ_i} \qquad \forall i = 1,..,n.
\end{array}
\end{equation}
The first equation is simply Complementary Slackness. The second equation is implied by the positive semidefinite cone constraint on $Z$ and $\Diag(w)-C$. The third equation rewrites the second equation isolating $w$ term. From the third equation, there are many ways to solve for $w_i$ values, and in fact the problem is separable in each $w_i$. The solution described in (\ref{eq:primal_solution_L}) uses least squares to compute $w_i$ values since $Z$ is often a numerical approximation to the optimal solution of (\ref{eq:compute_D_L_dual}).
\end{proof}

Computing a primal solution to the model (\ref{eq:compute_D_Q_dual}) from a dual solution is easier with a derivation from Lagrangian dual. 
\begin{lemma}
Let $Z$ be an optimal solution to (\ref{eq:compute_D_Q_dual}). Then a corresponding primal optimal solution can be computed by;
\begin{equation}
\label{eq:primal_solution_Q}
\begin{array}{ll@{}ll}
w_i = Z_{ii} \quad \forall i = 1,..,n.
\end{array}
\end{equation}
\end{lemma}
\begin{proof}
Writing the Lagrangian corresponding to (\ref{eq:compute_D_Q}) gives
\begin{equation}
\label{eq:compute_D_Q_lagrangian}
\begin{array}{ll@{}ll}
L(w,Z) = \frac{1}{2} w^Tw - \tr((\Diag(w)-C)Z).
\end{array}
\end{equation}
Minimizing over $w$ to get the dual function we get
\begin{equation}
\label{eq:compute_D_Q_lagrangian2}
\begin{array}{ll@{}ll}
w = \Diag(Z),
\end{array}
\end{equation}
which is exactly as described in (\ref{eq:primal_solution_Q}).
\end{proof}

\section{Additional Numerical Results}
\label{ap:additional_numerical}
In this section, we present the numerical findings from the CS experiments in detail. Algorithms are encoded in the table compactly as Newton + (Descent Condition Applied).  $\mathcal{D}$ denotes a dictionary of weights used in the CIWHT algorithm. Each $\mathcal{D}_i$ is explained in the Table \ref{tab:descent_dictionary}. Superscript denotes the period, i.e. how many times each of the weights in the CIWHT algorithm is applied before changing it for another. We remark again that Newton $+ \mathcal{D}_1^1$ denotes a complicated search algorithm that combines line search strategy, gradient-based restart and Newton steps for acceleration. 

\begin{table}[ht]
\centering
\begin{tabular}{|c|c|c|c|c|}
\hline
& \multicolumn{3}{|c|}{Order of Descent} \\ \hline
$\mathcal{D}_1$ & $D_Q$ & $D_L$ & $L$ \\ \hline
$\mathcal{D}_2$ & $D_L$ & $D_Q$ & $L$ \\ \hline
$\mathcal{D}_3$ & $D_L$ & $L$ & $D_Q$ \\ \hline
$\mathcal{D}_4$ & $D_Q$ & $L$ & $D_L$ \\ \hline
\end{tabular}
\caption{Order of descent conditions applied for each dictionary}
\label{tab:descent_dictionary}
\end{table}

\begin{table}[ht]
\centering
\begin{tabular}{|c|c|c|c|c|}
\hline
Sparsity &Newton + L &   Newton + $D_L$ &   Newton + $D_Q$ &   Newton + $\mathcal{D}_1^1$ \\
\hline
  5 &         99.8 &          100   &          100   &                        100   \\ \hline
  6 &        100   &          100   &          100   &                        100   \\ \hline
  7 &         99.8 &          100   &          100   &                        100   \\ \hline
  8 &        100   &          100   &          100   &                        100   \\ \hline
  9 &        100   &          100   &          100   &                        100   \\ \hline
 10 &        100   &          100   &          100   &                        100   \\ \hline
 11 &        100   &          100   &          100   &                        100   \\ \hline
 12 &        100   &          100   &          100   &                        100   \\ \hline
 13 &        100   &          100   &          100   &                        100   \\ \hline
 14 &        100   &          100   &          100   &                        100   \\ \hline
 15 &         99.8 &          100   &          100   &                        100   \\ \hline
 16 &        100   &          100   &          100   &                        100   \\ \hline
 17 &         99.8 &           99   &           99.4 &                         99.8 \\ \hline
 18 &        100   &           99.6 &           99.6 &                         99.6 \\ \hline
 19 &        100   &           98.8 &           99.2 &                         99.8 \\ \hline
 20 &         99.8 &           99   &           99.4 &                         99.8 \\ \hline
 21 &         99.4 &           98   &           98.4 &                         98.8 \\ \hline
 22 &         99   &           97.6 &           98.2 &                         98.8 \\ \hline
 23 &         98.2 &           96   &           95.8 &                         97.4 \\ \hline
 24 &         97.2 &           93.8 &           93.2 &                         95.8 \\ \hline
 25 &         95   &           89   &           92.8 &                         93.8 \\ \hline
 26 &         91.2 &           84.8 &           89   &                         90.4 \\ \hline
 27 &         84.2 &           79.6 &           81.2 &                         87.4 \\ \hline
 28 &         73   &           74   &           76.2 &                         81.4 \\ \hline
 29 &         60.2 &           58   &           63.4 &                         67   \\ \hline
 30 &         44   &           48   &           51.6 &                         60.8 \\ \hline
 31 &         24.8 &           35.4 &           34.8 &                         41.8 \\ \hline
 32 &         17.8 &           22.2 &           23.4 &                         28   \\ \hline
 33 &          4.4 &           13.6 &           12.2 &                         14.8 \\ \hline
 34 &          2   &            5.4 &            5.8 &                          7.4 \\ \hline
 35 &          0.8 &            1.2 &            1.6 &                          1.4 \\ \hline
\end{tabular}
\caption{Recovery rates for randomly generated sparse linear systems of size $64\times 256$ for changing sparsity levels.}
\label{tab:extra_numerical_results1}
\end{table}

\begin{table}[ht]
\centering
\begin{tabular}{|c|c|c|c|c|}
\hline
Sparsity &Newton + $\mathcal{D}_1^{10}$ &   Newton + $\mathcal{D}_2^1$ &  Newton + $\mathcal{D}_2^{10}$ &   Newton + $\mathcal{D}_3^1$ \\
\hline
  5 &                      100   &                        100   &                        100   &                        100   \\  \hline
  6 &                      100   &                        100   &                        100   &                        100   \\ \hline
  7 &                      100   &                        100   &                        100   &                        100   \\ \hline
  8 &                      100   &                        100   &                        100   &                        100   \\ \hline
  9 &                      100   &                        100   &                        100   &                        100   \\ \hline
 10 &                      100   &                        100   &                        100   &                        100   \\ \hline
 11 &                      100   &                        100   &                        100   &                        100   \\ \hline
 12 &                      100   &                        100   &                        100   &                         99.8 \\ \hline
 13 &                      100   &                        100   &                        100   &                        100   \\ \hline
 14 &                      100   &                        100   &                        100   &                        100   \\ \hline
 15 &                      100   &                        100   &                        100   &                        100   \\ \hline
 16 &                      100   &                        100   &                        100   &                         99.8 \\ \hline
 17 &                       99.8 &                         99.6 &                         99.8 &                         99.8 \\ \hline
 18 &                       99.8 &                        100   &                        100   &                        100   \\ \hline
 19 &                       99.6 &                         99.4 &                         99.6 &                         99.6 \\ \hline
 20 &                       99.8 &                         99.8 &                         99.4 &                        100   \\ \hline
 21 &                       98.4 &                         99   &                         99.4 &                         99.4 \\ \hline
 22 &                       99   &                         98.4 &                         98.8 &                         98.8 \\ \hline
 23 &                       96.6 &                         97   &                         97.8 &                         97.4 \\ \hline
 24 &                       96   &                         95.6 &                         96   &                         95.4 \\ \hline
 25 &                       95   &                         95.2 &                         93.4 &                         93.8 \\ \hline
 26 &                       92   &                         92.2 &                         91.8 &                         90.2 \\ \hline
 27 &                       85   &                         84.6 &                         85.6 &                         85.8 \\ \hline
 28 &                       83.6 &                         82.4 &                         81.2 &                         83.2 \\ \hline
 29 &                       67.6 &                         69.4 &                         66.4 &                         68.8 \\ \hline
 30 &                       56.8 &                         56.4 &                         55.6 &                         60   \\ \hline
 31 &                       41.8 &                         41.2 &                         40.6 &                         42.6 \\ \hline
 32 &                       27   &                         30.8 &                         26.2 &                         27.8 \\ \hline
 33 &                       14   &                         14.6 &                         15.4 &                         15.2 \\ \hline
 34 &                        6.8 &                          6.8 &                          7.4 &                          5.8 \\ \hline
 35 &                        1.6 &                          2.4 &                          1.2 &                          2.4 \\ \hline
\end{tabular}
\caption{Recovery rates for randomly generated problems (cont.).}
\label{tab:extra_numerical_results2}
\end{table}

\begin{table}[ht]
\centering
\begin{tabular}{|c|c|c|c|}
\hline
Sparsity &Newton + $\mathcal{D}_3^{10}$ &   Newton + $\mathcal{D}_4^1$ &  Newton + $\mathcal{D}_4^{10}$ \\
\hline
  5 &            100   &            100   &                     100   \\ \hline
  6 &            100   &            100   &                     100   \\ \hline
  7 &            100   &            100   &                     100   \\ \hline
  8 &            100   &            100   &                     100   \\ \hline
  9 &            100   &            100   &                     100   \\ \hline
 10 &            100   &            100   &                     100   \\ \hline
 11 &            100   &            100   &                     100   \\ \hline
 12 &            100   &             99.8 &                     100   \\ \hline
 13 &            100   &            100   &                     100   \\ \hline
 14 &            100   &            100   &                     100   \\ \hline
 15 &            100   &             99.8 &                     100   \\ \hline
 16 &            100   &            100   &                     100   \\ \hline
 17 &             99.4 &            100   &                      99.6 \\ \hline
 18 &            100   &             99.8 &                      99.8 \\ \hline
 19 &             99.8 &            100   &                      99.8 \\ \hline
 20 &             99.8 &             99.6 &                     100   \\ \hline
 21 &             99.6 &             99.2 &                      99.2 \\ \hline
 22 &             98.6 &             98.6 &                      98.4 \\ \hline
 23 &             97.6 &             97.6 &                      97.8 \\ \hline
 24 &             95.8 &             95   &                      96   \\ \hline
 25 &             94.6 &             94   &                      94.4 \\ \hline
 26 &             91.2 &             90.2 &                      91.2 \\ \hline
 27 &             83.6 &             88   &                      85.8 \\ \hline
 28 &             81   &             82.4 &                      80.6 \\ \hline
 29 &             70.4 &             70.8 &                      67.4 \\ \hline
 30 &             58.4 &             57.6 &                      56.4 \\ \hline
 31 &             40.4 &             42.6 &                      43.4 \\ \hline
 32 &             28   &             28   &                      26.8 \\ \hline
 33 &             15.2 &             16.2 &                      15.2 \\ \hline
 34 &              8.6 &              7   &                       7.8 \\ \hline
 35 &              3.2 &              2.6 &                       3.4 \\ \hline
\end{tabular}
\caption{Recovery rates for randomly generated problems (cont.).}
\label{tab:extra_numerical_results3}
\end{table}

\end{document}

%% file: ex_shared.tex
\usepackage{lipsum}
\usepackage{graphicx}
\usepackage{epstopdf}
\usepackage{outlines}
\usepackage{comment}
\usepackage{listings}
\usepackage{dirtytalk}
\usepackage{amssymb,amsfonts}
\usepackage[title]{appendix}
\usepackage{algorithmic} 
\usepackage{listing}

\usepackage{mathtools}

\usepackage{svg}

\usepackage{makecell}
\usepackage{pifont}

\usepackage{subcaption}

\ifpdf
  \DeclareGraphicsExtensions{.eps,.pdf,.png,.jpg}
\else
  \DeclareGraphicsExtensions{.eps}
\fi

\newcommand{\cmark}{\ding{51}}%
\newcommand{\xmark}{\ding{55}}%

\newcommand{\Ls}{\textit{L}}

\DeclarePairedDelimiter{\floor}{\lfloor}{\rfloor}

\newtheorem{assumption}{Assumption}

\newsiamremark{remark}{Remark}
\newsiamremark{hypothesis}{Hypothesis}
\crefname{hypothesis}{Hypothesis}{Hypotheses}
\newsiamthm{claim}{Claim}

\headers{Separable Bregman \& Sparsity}{F. S. Akta\c{s} and M. \c{C}.  Pinar}

\title{Separable Bregman Framework For Sparsity Constrained Nonlinear Optimization}

\author{Fatih S. Akta\c{s}\thanks{Department of Industrial Engineering and Operations Research, Columbia University, New York 
  (\email{fatih.aktas@columbia.edu}).}
\and Mustafa \c{C}.  Pinar\thanks{Corresponding author: Department of Industrial Engineering, Bilkent University, Ankara
  (\email{mustafap@bilkent.edu.tr}).}
}

\usepackage{amsopn}

\DeclareMathOperator{\Diag}{Diag}

\DeclareMathOperator*{\argmin}{arg\,min}

\DeclareMathOperator{\tr}{tr}

\DeclareMathOperator{\R}{\mathbb{R}}
\DeclareMathOperator{\interior}{int}
\DeclareMathOperator{\dom}{dom}

\makeatletter
\newcommand*{\addFileDependency}[1]{
  \typeout{(#1)}
  \@addtofilelist{#1}
  \IfFileExists{#1}{}{\typeout{No file #1.}}
}
\makeatother


%% file: main.bbl
\begin{thebibliography}{10}

\bibitem{absil_rtr}
{\sc P.-A. Absil, C.~Baker, and K.~Gallivan}, {\em Trust-region methods on riemannian manifolds}, Foundations of Computational Mathematics, 7 (2006), p.~303–330, \url{https://doi.org/10.1007/s10208-005-0179-9}, \url{http://dx.doi.org/10.1007/s10208-005-0179-9}.

\bibitem{agarwal2010fast}
{\sc A.~Agarwal, S.~Negahban, and M.~J. Wainwright}, {\em Fast global convergence rates of gradient methods for high-dimensional statistical recovery}, Advances in Neural Information Processing Systems, 23 (2010).

\bibitem{bpg_alternating_nmf}
{\sc M.~Ahookhosh, L.~T.~K. Hien, N.~Gillis, and P.~Patrinos}, {\em Multi-block bregman proximal alternating linearized minimization and its application to orthogonal nonnegative matrix factorization}, Computational Optimization and Applications, 79 (2021), pp.~681--715.

\bibitem{mosek}
{\sc M.~ApS}, {\em The MOSEK optimization toolbox for Python manual. Version 10.1.}, 2024, \url{https://docs.mosek.com/latest/pythonapi/index.html}.

\bibitem{bahmani2013greedy}
{\sc S.~Bahmani, B.~Raj, and P.~T. Boufounos}, {\em Greedy sparsity-constrained optimization}, The Journal of Machine Learning Research, 14 (2013), pp.~807--841.

\bibitem{bbcsicopt}
{\sc H.~Bauschke, J.~Borwein, and P.~Combettes}, {\em Bregman monotone optimization algorithms}, SIAM Journal on Control and Optimization, 42 (2003), pp.~596--636.

\bibitem{bregman_descent}
{\sc H.~H. Bauschke, J.~Bolte, and M.~Teboulle}, {\em A descent lemma beyond lipschitz gradient continuity: First-order methods revisited and applications.}, Math. Oper. Res., 42 (2017), pp.~330--348, \url{http://dblp.uni-trier.de/db/journals/mor/mor42.html#BauschkeBT17}.

\bibitem{Bauschke2017}
{\sc H.~H. Bauschke and P.~L. Combettes}, {\em Convex Analysis and Monotone Operator Theory in Hilbert Spaces}, Springer International Publishing, 2017, \url{https://doi.org/10.1007/978-3-319-48311-5}, \url{http://dx.doi.org/10.1007/978-3-319-48311-5}.

\bibitem{first_order_amir_beck}
{\sc A.~Beck}, {\em First-Order Methods in Optimization}, SIAM, 2017.

\bibitem{amir_beck_sparse_opt}
{\sc A.~Beck and Y.~C. Eldar}, {\em Sparsity constrained nonlinear optimization: Optimality conditions and algorithms}, SIAM Journal on Optimization, 23 (2013), pp.~1480--1509.

\bibitem{mfista}
{\sc A.~Beck and M.~Teboulle}, {\em Fast gradient-based algorithms for constrained total variation image denoising and deblurring problems}, IEEE Transactions on Image Processing, 18 (2009), pp.~2419--2434, \url{https://doi.org/10.1109/TIP.2009.2028250}.

\bibitem{berg-friedlander}
{\sc E.~Berg and M.~Friedlander}, {\em Sparse optimization with least squares constraints}, SIAM Journal on Optimization, 21 (2011), pp.~1201--1229.

\bibitem{bertsekas1997nonlinear}
{\sc D.~P. Bertsekas}, {\em Nonlinear programming}, Athena Scientific, 1995.

\bibitem{blumsignal}
{\sc T.~Blumensath}, {\em Accelerated iterative hard thresholding}, Signal Processing, 92 (2012), pp.~752--756.

\bibitem{Blumensath2012CompressedSW}
{\sc T.~Blumensath}, {\em Compressed sensing with nonlinear observations and related nonlinear optimization problems}, IEEE Transactions on Information Theory, 59 (2012), pp.~3466--3474, \url{https://api.semanticscholar.org/CorpusID:14200372}.

\bibitem{blumdaviesharmonic}
{\sc T.~Blumensath and M.~Davies}, {\em Iterative hard thresholding for compressed sensing}, Applied Computational and Harmonic Analysis, 27 (2009), pp.~265--274.

\bibitem{blumdavies-selected}
{\sc T.~Blumensath and M.~Davies}, {\em Normalized iterative hard thresholding: guaranteed stability and performance}, IEEE Journal on Selected Topics in Signal Processing, 4 (2010), pp.~298--309.

\bibitem{Blumensath-davies}
{\sc T.~Blumensath and M.~E. Davies}, {\em Iterative theresholding for sparse approximations}, Journal of Fourier Analysis and Applications, 14 (2008), pp.~629--654.

\bibitem{proximal_alternating_linearized}
{\sc J.~Bolte, S.~Sabach, and M.~Teboulle}, {\em Proximal alternating linearized minimization for nonconvex and nonsmooth problems}, Mathematical Programming, 146 (2014), pp.~459--494.

\bibitem{bregman_descent_phase_retrieval}
{\sc J.~Bolte, S.~Sabach, M.~Teboulle, and Y.~Vaisbourd}, {\em First order methods beyond convexity and lipschitz gradient continuity with applications to quadratic inverse problems}, SIAM Journal on Optimization, 28 (2018), pp.~2131--2151.

\bibitem{bolteteboulledynamic}
{\sc J.~Bolte and M.~Teboulle}, {\em Barrier operators and associated gradient-like dynamical systems for constrained minimization problems}, SIAM Journal on Control and Optimization, 42 (2003), pp.~1266--1292.

\bibitem{boumal2016non}
{\sc N.~Boumal, V.~Voroninski, and A.~Bandeira}, {\em The non-convex burer-monteiro approach works on smooth semidefinite programs}, Advances in Neural Information Processing Systems, 29 (2016).

\bibitem{bregman_original}
{\sc L.~Bregman}, {\em The relaxation method of finding the common point of convex sets and its application to the solution of problems in convex programming}, USSR Computational Mathematics and Mathematical Physics, 7 (1967), pp.~200--217, \url{https://doi.org/https://doi.org/10.1016/0041-5553(67)90040-7}, \url{https://www.sciencedirect.com/science/article/pii/0041555367900407}.

\bibitem{burer_monteiro}
{\sc S.~Burer and R.~D. Monteiro}, {\em A nonlinear programming algorithm for solving semidefinite programs via low-rank factorization}, Mathematical Programming, 95 (2003), p.~329–357, \url{https://doi.org/10.1007/s10107-002-0352-8}, \url{http://dx.doi.org/10.1007/s10107-002-0352-8}.

\bibitem{candes-eldar-needell}
{\sc E.~Cand\`es, Y.~Eldar, and D.~Needell}, {\em Compressed sensing with coherent and redundant dictionaries}, Applied Computational and Harmonic Analysis, 31 (2011), pp.~59--73.

\bibitem{candes2005l1}
{\sc E.~Cand{\`e}s, J.~Romberg, et~al.}, {\em l1-magic: Recovery of sparse signals via convex programming}, URL: www. acm. caltech. edu/l1magic/downloads/l1magic. pdf, 4 (2005), p.~16.

\bibitem{candesrip}
{\sc E.~J. Cand\`es}, {\em The restricted isometry property and its implications for compressed sensing}, Comptes Rendus Mathématiques, 346 (2008), pp.~589--592.

\bibitem{candes2006robust}
{\sc E.~J. Cand{\`e}s, J.~Romberg, and T.~Tao}, {\em Robust uncertainty principles: Exact signal reconstruction from highly incomplete frequency information}, IEEE Transactions on information theory, 52 (2006), pp.~489--509.

\bibitem{candes2005decoding}
{\sc E.~J. Cand\`es and T.~Tao}, {\em Decoding by linear programming}, IEEE transactions on information theory, 51 (2005), pp.~4203--4215.

\bibitem{censorzenios}
{\sc Y.~Censor and S.~Zenios}, {\em Proximal minimization algorithm with d-functions}, Journal of Optimization Theory and Applications, 73 (1992), pp.~451--464.

\bibitem{chen_teboulle1993}
{\sc G.~Chen and M.~Teboulle}, {\em Convergence analysis of a proximal-like minimization algorithm using bregman functions}, SIAM Journal on Optimization, 3 (1993), pp.~538--543, \url{https://doi.org/10.1137/0803026}, \url{https://doi.org/10.1137/0803026}, \url{https://arxiv.org/abs/https://doi.org/10.1137/0803026}.

\bibitem{diamond2016cvxpy}
{\sc S.~Diamond and S.~Boyd}, {\em {CVXPY}: {A} {P}ython-embedded modeling language for convex optimization}, Journal of Machine Learning Research, 17 (2016), pp.~1--5.

\bibitem{donoho2006compressed}
{\sc D.~L. Donoho}, {\em Compressed sensing}, IEEE Transactions on information theory, 52 (2006), pp.~1289--1306.

\bibitem{eckstein}
{\sc J.~Eckstein}, {\em Nonlinear proximal point algorithms using bregman functions, with applications to convex programming}, Mathematics of Operations Research, 18 (1993), pp.~202--226.

\bibitem{cnv_BM_SDP}
{\sc M.~A. Erdo\~gdu, A.~\"Ozda\~glar, P.~A. Parrilo, and N.~D. Vanli}, {\em Convergence rate of block-coordinate maximization burer–monteiro method for solving large sdps}, Mathematical Programming, 195 (2021), p.~243–281, \url{https://doi.org/10.1007/s10107-021-01686-3}, \url{http://dx.doi.org/10.1007/s10107-021-01686-3}.

\bibitem{figueiredo2007gradient}
{\sc M.~A. Figueiredo, R.~D. Nowak, and S.~J. Wright}, {\em Gradient projection for sparse reconstruction: Application to compressed sensing and other inverse problems}, IEEE Journal of selected topics in signal processing, 1 (2007), pp.~586--597.

\bibitem{foucart-siam}
{\sc S.~Foucart}, {\em Hard thresholding pursuit: an algorithm for compressive sensing}, SIAM Journal on Numerical Analysis, 49 (2011), pp.~2543--2563.

\bibitem{foucart-rauhut}
{\sc S.~a. Foucart and H.~Rauhut}, {\em A mathematical introduction to compressive sensing}, vol.~44, Springer, 2013.

\bibitem{javanmard2016phase}
{\sc A.~Javanmard, A.~Montanari, and F.~Ricci-Tersenghi}, {\em Phase transitions in semidefinite relaxations}, Proceedings of the National Academy of Sciences, 113 (2016), pp.~E2218--E2223.

\bibitem{journee2010low}
{\sc M.~Journ{\'e}e, F.~Bach, P.-A. Absil, and R.~Sepulchre}, {\em Low-rank optimization on the cone of positive semidefinite matrices}, SIAM Journal on Optimization, 20 (2010), pp.~2327--2351.

\bibitem{block_bregman_bpg_nmf}
{\sc L.~T. Khanh~Hien, D.~N. Phan, N.~Gillis, M.~Ahookhosh, and P.~Patrinos}, {\em Block bregman majorization minimization with extrapolation}, SIAM Journal on Mathematics of Data Science, 4 (2022), pp.~1--25.

\bibitem{bi_smooth_nmf}
{\sc Q.~Li, Z.~Zhu, G.~Tang, and M.~B. Wakin}, {\em Provable bregman-divergence based methods for nonconvex and non-lipschitz problems}, arXiv preprint arXiv:1904.09712,  (2019).

\bibitem{relatively_smooth_opt}
{\sc H.~Lu, R.~M. Freund, and Y.~Nesterov}, {\em Relatively smooth convex optimization by first-order methods, and applications}, SIAM Journal on Optimization, 28 (2018), pp.~333--354.

\bibitem{lu-cone}
{\sc Z.~Lu}, {\em Iterative hard thresholding methods forl0 regularized convex cone programming}, Mathematical Programming, 147 (2014), pp.~125--154.

\bibitem{lu-zhang}
{\sc Z.~Lu and Y.~Zhang}, {\em Sparse approximation via penalty decomposition methods}, SIAM Journal on Optimization, 23 (2013), pp.~2448--2478.

\bibitem{mei2017solvingsdpssynchronizationmaxcut}
{\sc S.~Mei, T.~Misiakiewicz, A.~Montanari, and R.~I. Oliveira}, {\em Solving sdps for synchronization and maxcut problems via the grothendieck inequality}, 2017, \url{https://arxiv.org/abs/1703.08729}, \url{https://arxiv.org/abs/1703.08729}.

\bibitem{spca_np_hard}
{\sc B.~Moghaddam, Y.~Weiss, and S.~Avidan}, {\em Generalized spectral bounds for sparse {LDA}}, in Proceedings of the 23rd International Conference on Machine Learning, ICML '06, New York, NY, USA, 2006, Association for Computing Machinery, p.~641–648.

\bibitem{sparse_ls_np_hard}
{\sc B.~Natarajan}, {\em Sparse approximate solutions to linear systems}, SIAM Journal on Computing, 24 (1995), pp.~227--234.

\bibitem{needell2009cosamp}
{\sc D.~Needell and J.~A. Tropp}, {\em Cosamp: Iterative signal recovery from incomplete and inaccurate samples}, Applied and computational harmonic analysis, 26 (2009), pp.~301--321.

\bibitem{high_dimensional_decomposable_regularizer}
{\sc S.~N. Negahban, P.~Ravikumar, M.~J. Wainwright, and B.~Yu}, {\em {A Unified Framework for High-Dimensional Analysis of $M$-Estimators with Decomposable Regularizers}}, Statistical Science, 27 (2012), pp.~538 -- 557, \url{https://doi.org/10.1214/12-STS400}, \url{https://doi.org/10.1214/12-STS400}.

\bibitem{nesterov2018lectures}
{\sc Y.~Nesterov et~al.}, {\em Lectures on convex optimization}, vol.~137, Springer, 2018.

\bibitem{nocedal1999numerical}
{\sc J.~Nocedal and S.~J. Wright}, {\em Numerical optimization}, Springer, 1999.

\bibitem{pan-luo-xiu}
{\sc L.~Pan, Z.~Luo, and N.~Xiu}, {\em Restricted robinson constraint qualification and optimality for cardinality-constrained cone programming}, Journal of Optimization Theory and Applications, 175 (2017), pp.~104--118.

\bibitem{pan-xiu-zhou}
{\sc L.~Pan, N.~Xiu, and S.~Zhou}, {\em On solutions of sparsity constrained optimization}, Journal of the Operations Research Society of China, 3 (2015), pp.~421--439.

\bibitem{pati1993orthogonal}
{\sc Y.~C. Pati, R.~Rezaiifar, and P.~S. Krishnaprasad}, {\em Orthogonal matching pursuit: Recursive function approximation with applications to wavelet decomposition}, in Proceedings of 27th Asilomar conference on signals, systems and computers, IEEE, 1993, pp.~40--44.

\bibitem{shalev2010trading}
{\sc S.~Shalev-Shwartz, N.~Srebro, and T.~Zhang}, {\em Trading accuracy for sparsity in optimization problems with sparsity constraints}, SIAM Journal on Optimization, 20 (2010), pp.~2807--2832.

\bibitem{teboullemor}
{\sc M.~Teboulle}, {\em Entropic proximal mappings with applications to nonlinear programming}, Mathematics of Operations Research, 17 (1992), pp.~670--690.

\bibitem{nmf_bregman_teboulle_yakov}
{\sc M.~Teboulle and Y.~Vaisbourd}, {\em Novel proximal gradient methods for nonnegative matrix factorization with sparsity constraints}, SIAM Journal on Imaging Sciences, 13 (2020), pp.~381--421, \url{https://doi.org/10.1137/19M1271750}, \url{https://doi.org/10.1137/19M1271750}, \url{https://arxiv.org/abs/https://doi.org/10.1137/19M1271750}.

\bibitem{teboulle2023elementary}
{\sc M.~Teboulle and Y.~Vaisbourd}, {\em An elementary approach to tight worst case complexity analysis of gradient based methods}, Mathematical Programming, 201 (2023), pp.~63--96.

\bibitem{tseng_convex_concave}
{\sc P.~Tseng}, {\em On accelerated proximal gradient methods for convex-concave optimization}, tech. report, MIT, Cambridge, MA, 2008.

\bibitem{mixing_method}
{\sc P.-W. Wang, W.-C. Chang, and J.~Z. Kolter}, {\em The mixing method: low-rank coordinate descent for semidefinite programming with diagonal constraints}, 2018, \url{https://arxiv.org/abs/1706.00476}, \url{https://arxiv.org/abs/1706.00476}.

\bibitem{xie-li-liang}
{\sc S.~Xie, J.~Li, and K.~Liang}, {\em k-sparse vector recovery via $\ell_1 - \alpha\ell_2$ local minimization}, Journal of Optimization Theory and Applications, 201 (2024), pp.~75--102.

\bibitem{prox-iht}
{\sc F.~Yang, Y.~Shen, and Z.~Liu}, {\em The proximal alternating iterative hard thresholding method for l0 minimization, with complexity $o(1/\sqrt{k})$}, Journal of Computational and Applied Mathematics, 311 (2017), pp.~115--129.

\bibitem{yang2011alternating}
{\sc J.~Yang and Y.~Zhang}, {\em Alternating direction algorithms for $\backslash$ell\_1-problems in compressive sensing}, SIAM journal on scientific computing, 33 (2011), pp.~250--278.

\bibitem{yinyu_lecture_notes}
{\sc Y.~Ye}, {\em Conic linear programming}.
\newblock Available at \url{https://web.stanford.edu/class/msande314/}, December 2004.
\newblock Revised, October 2017.

\bibitem{yuan_grad_hard_threshold}
{\sc X.-T. Yuan, P.~Li, and T.~Zhang}, {\em Gradient hard thresholding pursuit}, Journal of Machine Learning Research, 18 (2018), pp.~1--43.

\bibitem{yurtsever19}
{\sc A.~Yurtsever, O.~Fercoq, and V.~Cevher}, {\em A conditional-gradient-based augmented lagrangian framework}, in Proc. 36th Int. Conf. Machine Learning (ICML), 2019.

\bibitem{yurtsevercgal}
{\sc A.~Yurtsever, J.~A. Tropp, O.~Fercoq, M.~Udell, and V.~Cevher}, {\em Scalable semidefinite programming}, SIAM Journal on Mathematics of Data Science, 3 (2021), pp.~171--200.

\bibitem{gpnp}
{\sc S.~Zhou}, {\em Gradient projection newton pursuit for sparsity constrained optimization}, Applied and Computational Harmonic Analysis, 61 (2022), pp.~75--100.

\bibitem{nhtp}
{\sc S.~Zhou, N.~Xiu, and H.-D. Qi}, {\em Global and quadratic convergence of newton hard-thresholding pursuit}, The Journal of Machine Learning Research, 22 (2021), pp.~599--643.

\end{thebibliography}
